\documentclass[a4,10pt]{amsart}
\usepackage{hyperref}
\usepackage{amsfonts,amsmath,amssymb,amsthm}
\usepackage{fullpage, setspace, color}
\usepackage{indentfirst}
\usepackage{graphicx}
\usepackage{subfigure}
\usepackage{float}
\newtheorem{theorem}{Theorem}[section]
\newtheorem{lemma}[theorem]{Lemma}

\newtheorem{thm}{Theorem}[section]
\newtheorem{cor}[thm]{Corollary}
\newtheorem{rem}[thm]{Remark}

\newtheorem{remark}[thm]{Remark}
\newtheorem{definition}{Definition}[section]
\numberwithin{equation}{section}

\newcommand{\Ga}{\Gamma}

\def\<{\left\langle} \def\>{\right\rangle}
\def\({\left(} \def\){\right)}

\usepackage{setspace}


\title
{Spatial and temporal dynamics of an  almost 
	periodic reaction-diffusion system for West 
	Nile virus}


\email{chengchengcheng@amss.ac.cn}
\email{zhzheng@amt.ac.cn}

\thanks{The authors are supported by the NSF of China(No. 11671382,12031020), 
CAS Key 
Project of Frontier Sciences(No. QYZDJ-SSW-JSC003), the Key Lab. of Random 
Complex Structures and Data Sciences CAS and National Center for Mathematics 
and Interdisciplinary Sciences CAS}

\thanks{$^*$ Corresponding author: Zuohuan Zheng}
\begin{document}
	\maketitle
	
	\centerline{ Chengcheng Cheng}
	\medskip
	{\footnotesize
		\centerline{Academy of Mathematics and Systems Science, Chinese Academy of Sciences}
		\centerline{Beijing 100190, China}
		\centerline{School of Mathematical Sciences, University of Chinese 
		Academy of Sciences}
		\centerline{Beijing 100049, China}
	} 
	
	\medskip
	
	\centerline{ Zuohuan Zheng$^*$}
	\medskip
	{\footnotesize
		\centerline{Academy of Mathematics and Systems Science, Chinese Academy of Sciences}
		\centerline{Beijing 100190, China}
		\centerline{School of Mathematical Sciences, University of Chinese 
			Academy of Sciences}
		\centerline{Beijing 100049, China}
		\centerline{College of Mathematics and Statistics, Hainan Normal 
		University}
		\centerline{Haikou, Hainan 571158, China}
		
	}
	
	\bigskip
	
	
	\begin{abstract}
	{In current paper, we put forward a reaction-diffusion system for West Nile 
		virus in spatial 
		heterogeneous and time
		almost periodic environment
		with 
		free boundaries to investigate the influences of the 
		habitat differences and seasonal 
		variations on the propagation of West 
		Nile 
		virus. The existence, uniqueness and regularity estimates of the global 
		solution for this disease model are given. 
		Focused on the 
		effects of
		spatial heterogeneity and time almost periodicity, we apply the 
		principal Lyapunov exponent $\lambda(t)$ with time $t$ to get the 
		initial 
		infected domain threshold $L^*$ to
		analyze the long-time 
		dynamical behaviors of the 
		solution for this 
		almost periodic West 
		Nile virus  model and
		give the 
		spreading-vanishing dichotomy regimes of the disease.  Especially, we 
		prove 
		that the solution for this West Nile virus model converges to a time 
		almost 
		periodic function locally uniformly for  $x$ in $\mathbb R$ when the 
		spreading occurs, which is driven by spatial differences and seasonal 
		recurrence.  Moreover, the initial disease 
		infected domain and the front expanding 
		rate have  
		momentous impacts on the permanence and extinction of the epidemic 
		disease. Eventually, numerical simulations identify our theoretical 
		results. }
	\end{abstract}

	\section{Introduction}\label{s1}
	\noindent
West Nile virus (WNv) causes 
mosquito-borne 
epidemic diseases 
seriously 
threatening 
people's 
lives by invading people's nervous system. Since the West Nile virus broke out 
in New York 
in 1999, it has become endemic 
all over the United States. It was estimated that 7 million human infections 
had
occurred from 1999 to 2019, making it the main mosquito-borne virus infectious 
disease in 
America (\cite{plos2019}).  In recent years, the
infection of WNv has spread from 
North America 
to Europe, bringing about considerable deaths. In order to 
supply feasible measures 
to prevent and control 
the propagations of WNv, it is pretty worthwhile to apply 
mathematical models to
investigate the long-time spreading dynamics of the WNv.

West Nile virus spreads mainly through mosquitoes as the vectors and
biting birds as the hosts. Concentrated on the temporal 
transmission of the WNv, there have been many works by ordinary 
differential equations to explore the 
existence and 
stability of the equilibrium, and  introduce the basic reproduction number 
as a 
threshold value to study the transmission dynamics of WNv, 
such as \cite{wonham2004epidemic, cruz2005, bowman2005mathematical, 
	abdelrazec2014, chen2016model} and 
references therein. 

The free migration movements of the infected bird populations and mosquitoes 
populations 
are usually random, so 
the spatial
diffusion term should be in consideration. Therefore, only using ordinary 
differential systems to describe the spatial propagation of the West Nile virus 
is no more suitable.
In view of the spatial 
heterogeneity, Allen et al.~\cite{allen2008asymptotic} studied the 
following SIS reaction-diffusion model in 2008,
\begin{equation}\label{system}
	\left\{\begin{array}{ll}
		\frac{\partial {S}}{\partial t}=d_{S} \Delta {S}-\frac{\beta(x) 
			{S} {I}}{{S}+{I}}+\gamma(x) {I}, & x \in \Omega, \quad 
		t>0, \\
		\frac{\partial {I}}{\partial t}=d_{I} \Delta {I}+\frac{\beta (x)
			{S} {I}}{{S}+{I}}-\gamma (x){I}, & x \in \Omega, 
		\quad t>0,\\ \frac{\partial {S}}{\partial n}=\frac{\partial 
			{I}}{\partial n}=0, & x \in \partial \Omega, \quad t>0,
	\end{array}\right.
\end{equation}
where domain ${\Omega}\in \mathbb R^k(k\geq 1)$ is bounded with smooth boundary 
$\partial 
\Omega$; ${S}(x, t)$ and ${I}(x, t)$ are the 
population 
densities of 
susceptible 
and infected individuals at position $x$ and time $t$; positive constants 
$d_{S}$ and $d_{I}$ 
represent diffusion rates for the susceptible and infected 
populations;  $\beta(x)$ is the disease 
transmission rate at position $x$ and $\gamma(x)$ is the disease recovery rate 
at position $x$, 
both of which are positive 
H$\ddot{\rm o}$lder continuous 
functions, 
respectively. They studied the effects of
the heterogeneous media and the individual movement
of susceptible and infected populations on the permanence and 
eradication of the disease and obtained the global dynamics of model 
(\ref{system}) by basic reproduction number $R_0$.
Lewis et 
al.~\cite{lewis2006traveling} investigated the spreading speed of the WNv 
by a reaction-diffusion system. Maidana and Yang 
~\cite{maidana2009spatial} 
used the traveling wave solution of the WNv model to study the spatial 
spreading 
of the disease across North America. 

The 
infected boundaries driven by birds and mosquitoes migrating from one habitat  
to another change
with respect to time. Thus, applying the fixed studying domain is not 
appropriate.
Free boundaries conditions have largely attracted lots of concentrations 
recently 
and they
are 
frequently used 
in  biological mathematical models, for instance, 
\cite{chen2003free, du2010spreading, 
	xiao2014jde,  wang2014some, 
	guo2019, liu2019dcdsb,zhang2019}. In view of the moving
infected boundaries, Lin 
and Zhu \cite{lin2017spatial} investigated a 
reaction-diffusion system  to explore the spatial spreading of WNv using 
free boundaries to represent the disease spreading fronts. Tarboush et 
al.\cite{tarbou2017} studied a WNv model which incorporates a Partial 
differential equation and an ordinary differential equation with moving 
boundaries. 
Cheng and 
Zheng~\cite{cheng2020} considered a reaction-advection-diffusion WNv 
model with double free boundaries and studied the influence of advection 
term on the boundary asymptotic spreading speeds.


In reality, the outbreak of the disease is not always caused by 
single factor. 
Apart from the 
spatial heterogeneity, the temporal heterogeneity caused by 
alternations of seasonality is also a significant
factor in influencing the propagation of the disease. Peng and Zhao 
\cite{zhao2012sis} investigated the model $(\ref{system})$ in a 
time-periodic 
heterogeneous environment which the transmission rate $\beta(x,t)$ and 
recovery rate $\gamma(x,t)$ are periodic for time $t$.
Zhang and Wang~\cite{wang2016SIR} studied a diffusive 
SIR time periodic system and investigated the spatial dynamics of this epidemic 
model. 
Shan et al.~\cite{zhu2019} investigated a periodic compartmental 
WNv model
with time delay and obtained the effects of seasonal recurrent phenomena on the 
speading and recurrence of the epidemic disease.

From a biological view, the effects of the alternation of seasons on the 
disease
transmission rate, disease recovery rate and the disease death rate are not 
same.  Thus, these parameter periods for the epidemic model are usually 
different. 
Therefore, we had to look for more reasonable mathematical model.
Considering the differences of the periodic coefficients, it is
significant to study the time almost
periodic system. Shen and 
Yi~\cite{shen1998models} studied the convergence of the positive solution 
for almost periodic models of Fisher and Kolmogorov  type. Huang and Shen 
\cite{shen2009kpp} investigated the spreading 
dynamics of KPP models in time almost periodic and space periodic
environment and gave the estimates of the spreading speed. Wang and 
Zhao~\cite{zhao2013basic} discussed 
the basic reproduction ratio $R_0$ and obtained its computating formula for
almost periodic compartmental ordinary differential epidemic models. Wang et 
al.~\cite{wang2015sis} investigated a reaction-diffusion SIS model in a 
time almost periodic environment and discussed the influences of the basic 
reproduction number $R_0$ on the persistence or extinction of the solution for 
epidemic model. Recently, Qiang et al.~\cite{qiang2020} studied a nonlocal 
reaction-diffusion 
model with time delay in almost periodic media and discussed the threshold 
dynamics using the upper Lyapunov exponent.

However, there are few 
studies on mosquito-borne diseases using 
almost periodic system. For the sake of better exploring the mechanisms of 
the disease outbreak and more reasonably describing
the transmission rules of WNv, almost periodic mathematical 
biology models incorporate spatial heterogeneity with time almost periodicity 
should be vitally considered to study the propagation of 
WNv. 
Motivated by the previous studies,  we 
investigate the following WNv 
model with double free boundaries in spatial heterogeneous and time almost 
periodic media,
\begin{equation}\label{equation1}
	\left\{\begin{array}{ll}U_t=D_{1} U_{xx}+\alpha_{1}(x,t) \beta 
		\frac{N_{1}-U}{N_{1}} V-\gamma(x,t) U, & g(t)<x<h(t), t>0, \\ V_t=D_{2} 
		V_{xx}+\alpha_{2}(x,t) \beta \frac{N_{2}-V}{N_{1}} U-d(x,t) V, & 
		g(t)<x<h(t), t>0,\\ U(x, t)=V(x, t)=0, & x=h(t) \text { or } x=g(t), 
		t>0, 
		\\h(0)=h_{0},\quad 
		h^{\prime}(t)=-\mu U_{x}(h(t), t), & t>0, \\ 
		g(0)=-h_{0}, \quad g^{\prime}(t)=-\mu U_{x}(g(t), t), & t>0,\\ U(x, 
		0)=U_{0}(x), V(x, 
		0)=V_{0}(x), & -h_{0} \leq x \leq h_{0},\end{array}\right.
\end{equation}
where $U(x,t)$ and $V(x, t)$ are the densities of infected bird populations
and mosquito populations at location $x$ and time $t$, respectively; $N_1$ and 
$N_2$ 
are 
the total population capacities of the birds and mosquitoes; $D_{1}$ and 
$D_{2}$ 
are the diffusion rates of the birds and mosquitoes, respectively; 
$\alpha_{1}(x,t)$ and 
$\alpha_{2}(x,t)$ are the WNv transmission probabilities per bite to 
birds and 
mosquitoes at location $x$ and time $t$; $\beta$ is the biting rate of 
mosquitoes to birds; $\gamma(x,t)$ is 
the 
recovery rate of birds from infection at location $x$ and time $t$; $d(x,t)$ is 
the death rate of the 
mosquitoes at location $x$ and time $t$. The moving region $(g(t),h(t))$ 
is the infected domain of WNv. Meanwhile, we suppose that the double free 
boundaries 
submit to classical Stefan conditions obeying the Fick's first law, that is,  
$g^{\prime}(t)=-\mu 
U_{x}(g(t), t)$ and  $ 
h^{\prime}(t)=-\mu U_{x}(h(t), t),$ where $\mu$ is positive. Moreover, we 
assume that 
$\alpha_1(x,t)$, $\alpha_2(x,t)$, $\gamma(x, t)$, 
$d(x,t)\in 
C^{2+\alpha_0,1+\frac{\alpha_0}{2}}(\mathbb R\times [0,\infty))$ are positive 
bounded functions for some $\alpha_0\in(0,1)$, and uniformly almost
periodic in $t$. What is more, $\alpha_1(x,t)$, $\alpha_2(x,t)$, $\gamma(x, 
t)$, 
$d(x,t)$ have positive supper and lower bound.

In order to simplify the number of parameters in this model, denote
\begin{equation}\label{1.2}
	\begin{aligned}
		&a_{1}(x,t):=\frac{\alpha_{1}(x,t) \beta}{N_{1}}, 
		a_{2}(x,t):=\frac{\alpha_{2}(x,t) 
			\beta}{N_{1}},d_1(x,t):=\gamma(x,t), d_2(x,t):=d(x,t),
	\end{aligned}
\end{equation}
then $a_1(x,t),a_2(x,t),d_1(x,t),d_2(x,t)\in 
C^{2+\alpha,1+\frac{\alpha}{2}}(\mathbb R\times [0,\infty))$ for any $\alpha\in 
(0,\alpha_0).$

On the basis of the previous simplifications and assumptions, we are going to 
investigate the 
the following simplified WNv system,
\begin{equation}\label{system-1}
	\left\{\begin{array}{ll}{U_{t}=D_{1} U_{x x}+a_{1}(x,t)\left(N_{1}-U\right) 
			V-d_1(x,t) U,} & {g(t)<x<h(t), \enspace t>0,} \\ {V_{t}=D_{2} V_{x 
				x}+a_{2}(x,t)\left(N_{2}-V\right) U-d_2(x,t) V,} & 
		{g(t)<x<h(t), \enspace 
			t>0,} \\ {U(x, t)=V(x, t)=0,} & {x=h(t) \text { or } 
			x=g(t),\enspace t>0,} 
		\\ {h(0)=h_{0},\enspace h^{\prime}(t)=-\mu U_{x}(h(t), t),} & {t>0,} \\ 
		{g(0)=-h_{0}, \enspace g^{\prime}(t)=-\mu U_{x}(g(t), t),} & {t>0,} \\ 
		{U(x, 0)=U_{0}(x), \quad V(x, 0)=V_{0}(x),} & {-h_{0} \leq x \leq 
			h_{0}.}\end{array}\right.
\end{equation}
For the convenience of studying, we make the following assumptions about the 
initial functions 
$U_{0}$ and 
$V_{0}$,
\begin{equation}\label{system-2}
	\left\{\begin{array}{ll}{U_{0}(x) \in C^{2}(\left[-h_{0}, h_{0}\right]),} & 
		{U_{0}(\pm h_{0})=0, \enspace 0<U_{0}(x) \leq N_{1} \text { in 
			}\left(-h_{0}, h_{0}\right),} \\ {V_{0}(x) \in C^{2}(\left[-h_{0}, 
			h_{0}\right]),} & {V_{0}(\pm h_{0})=0,\enspace 0<V_{0}(x) \leq 
			N_{2} \text 
			{ in }\left(-h_{0}, h_{0}\right).}\end{array}\right.
\end{equation}

In this paper, our primary purpose is to research 
a reaction-diffusion WNv model with moving infected regions $(g(t), h(t))$ 
in 
the spatial heterogeneous and time almost 
periodic 
media, and discuss the effects of the spatial heterogeneity and time almost 
periodicity on the spreading and vanishing of the epidemic disease. In view of 
the biological reality, this WNv model (\ref{system-1}) is first 
proposed to incorporate the spatial heterogeneity with time almost periodicity 
in studying epidemic disease. 
We first 
give the 
global 
existence, uniqueness and regularity estimates of the solution, 
the method of which is not trivially similar to 
homogeneous WNv models (See Theorems \ref{existence},  
\ref{lemma-existence1} and \ref{lemma-existence2}).
In view of spatial variants 
with seasonal changes, the virus transmission rate 
($\alpha_1(x,t),\alpha_2(x,t)$) between 
mosquitoes and birds, birds recovery rate ($\gamma(x,t)$) from infection 
and mosquitoes death rate ($d(x,t)$) all depend on 
location $x$ 
and time $t$, which is more 
consistent with the disease spreading reality. Moreover, since the coefficients 
are heterogeneous and the boundary is moving, we introduce the 
principal Lyapunov 
exponent $\lambda(t)$ 
with respect to time $t$ (See section \ref{basic}) to get the initial infected 
domain $L^*$
as a threshold value and we obtain the spreading-vanishing dichotomy regimes of 
West Nile virus (See
Theorem \ref{dichotomy}) using it.  We prove that the eventually infected 
domain is no more than $2L^*$ when the vanishing occurs. Importantly, we prove 
that the 
solution 
for system 
(\ref{system-1}) converges to a time almost periodic function for fixed $x$ in 
bounded subsets of $\mathbb R$ 
when the spreading 
occurs, whose asymptotic behavior is very different from other homogeneous WNv 
models, the 
solution of which 
converges to a positive constant equilibrium, such as 
\cite{wonham2004epidemic}, \cite{lewis2006traveling},
\cite{lin2017spatial}. Our results show that the spatial heterogeneity and 
temporal almost peridocity driven by 
spatial differences and seasonal recurrence lead to the cyclic apperance of 
the cases of infection. Moreover, the initial 
West Nile virus infected domain and the front expanding rate have momentous 
impacts on the permanence and 
extinction of the epidemic disease.  These results are useful for 
people to 
understand the spreading dynamics of the disease with spatial and seasonal 
diversity and implement measures to prevent and control the transmission of the 
infectious diseases.  

The rest of the paper is arranged as follows. In section \ref{prelinminary}, we 
first
prepare some 
preliminaries and assumptions, then present the main results.   In section 
\ref{result1}, 
we provide  a detailed  proof of the  global 
existence, uniqueness and regulaity estimates of the solution for problem 
(\ref{system-1}) in the
time almost periodic and spatial heterogeneous environment which is not 
trivial. In 
section \ref{basic}, considering the spatial heterogeneity and time almost 
periodicity, we introduce the 
principal Lyapunov exponent and obtain some vital properties of this threshold 
value.
In section \ref{result2}, we apply the principal Lyapunov exponent to obtain 
the sufficient conditions for persistence 
and 
eradication 
of the epidemic disease and explore 
the long-time asymptotic behaviors of the solution for heterogeneous system 
(\ref{system-1}) by
applying a different method from other homogeneous and periodic 
WNv models. In section \ref{discussion}, we make several numerical simulations 
to identify our theory results, then present some discussions and 
biological meanings 
about our analytical results for WNv model.

\section{Preliminaries and Main Results}\label{prelinminary}
In the section, we make some  
preparations and display our main results.
\subsection{Preliminaries}
\noindent
\\Firstly, we recall several
definitions about almost periodic function from Section 2.1 of
~\cite{shen1998models} or Section 3 of~\cite{zhao2003}.

\begin{definition}
	(\romannumeral 1)	A function $f\in C(\mathbb{R},\mathbb{R}^k)(k
	\geq1)$ is called  
	an {almost 
		periodic} function
	if for any $\epsilon>0,$ the set 
	\[T(f,\epsilon):=\{\tau\in\mathbb R\mid |f(t+\tau)-f(t)|<\epsilon \text{ 
		for 
		any}\enspace t\in 
	\mathbb R\}\]
	is relatively dense in $\mathbb R$.	
	We say a matrix function $A(x,t)$ is almost periodic if every entry of it 
	is 
	almost periodic.
	\\(\romannumeral2)	A function $f\in C(\mathbb R\times \mathbb R,\mathbb 
	R)$ is  
	{uniformly 
		almost periodic} in $t$ if $f(x,\cdot)$ is almost periodic for every 
	$x\in \mathbb R$, and $f$ is uniformly continuous on $E\times \mathbb 
	R$ for 
	any compact set $E\subset \mathbb R$.
	\\(\romannumeral 3) A function $f(x,t,u,v)\in C(\mathbb R\times \mathbb 
	R\times \mathbb 
	R^m\times \mathbb R^n,\mathbb R^k)(m,n,k\geq1)$ is  {uniformly 
		almost 
		periodic} in 
	$t$ with 
	$x\in \mathbb R$ and $(u,v)$ in bounded sets if $f$ is uniformly 
	continuous for $t\in \mathbb R,x\in\mathbb R$ and $(u,v)$ in bounded sets 
	and  
	$f(x,t,u,v)$ is almost periodic in $t$ for every $x\in \mathbb R,$ 
	$u\in 
	\mathbb R^m$ and $v\in \mathbb R^n$.	
\end{definition} 
\begin{definition}
	(\romannumeral 4) The hull of a uniformly almost periodic matrix $A(x,t)$ 
	is defined 
	by 
	\[\textit{H}(A)=\{B(\cdot,\cdot)\mid\exists\enspace t_n\rightarrow\infty, 
	{\rm such 
		\enspace that} \enspace
	B(x,t+t_n)\rightarrow A(x,t) \enspace {\rm uniformly}\enspace  {\rm 
		for}\enspace  
	t\in\mathbb R,\]
	\[ x\enspace {\rm in\enspace bounded \enspace sets}\}.\]
	(\romannumeral 5) The hull of a uniformly almost periodic matrix 
	$F(x,t,u,v)$ is 
	defined by 
	\[\textit{H}(F)=\{G(\cdot,\cdot,\cdot,\cdot)\mid\exists\enspace 
	t_n\rightarrow\infty,{\rm  such 
		\enspace that }\enspace
	G(x,t+t_n,u,v)\rightarrow F(x,t,u,v)\] \[{\rm uniformly} \enspace {\rm 
		for}\enspace  
	t\in \mathbb R, (x,u,v)
	\enspace {\rm in\enspace bounded \enspace sets}\}.\]
\end{definition}

In general, the system $(\ref{system-1})$ can be seen as the special form of 
the 
following system,
\begin{equation}\label{system-3}
	\left\{\begin{array}{ll}{u_{t}=D_{1} u_{x x}+f_1(x,t,u,v),} & {g(t)<x<h(t), 
			\enspace t>0,} \\ {v_{t}=D_{2} v_{x 
				x}+f_2(x,t,u,v),} & {g(t)<x<h(t), \enspace 
			t>0,} \\ {u(x, t)=v(x, t)=0,} & {x=h(t) \text { or } 
			x=g(t),\enspace 
			t>0,} 
		\\ {h(0)=h_{0},\enspace h^{\prime}(t)=-\mu u_{x}(h(t), t),} & {t>0,} \\ 
		{g(0)=-h_{0}, \enspace g^{\prime}(t)=-\mu u_{x}(g(t), t),} & {t>0,} \\ 
		{u(x, 0)=u_{0}(x), \quad v(x, 0)=v_{0}(x),} & {-h_{0} \leq x \leq 
			h_{0},}\end{array}\right.
\end{equation}
where initial data $(u_0,v_0)$ satisfy (\ref{system-2}),
and $f_i(x,t,u,v)$ satisfies the following 
conditions for $i=1,2$.
\\
\textit{$(\textbf{H1})$\label{h1} $f_i(x,t,u,v)\in C^1(\mathbb 
	R^4)$,  
	$Df_i(x,t,u,v)=(\frac{\partial f_i}{\partial x}, \frac{\partial 
		f_i}{\partial t},\frac{\partial f_i}{\partial u},\frac{\partial 
		f_i}{\partial v})$ is bounded for $(x,t)\in 
	\mathbb R \times \mathbb R$ and  $(u,v)$ in 
	bounded sets. \\
	$(\textbf{H2})$ \label{h2}There 
	exist positive constants
	$M$ and $N$ such 
	that \[\sup_{t\in\mathbb R,x\in\mathbb R,\atop u\geq M,v\in \mathbb 
		R}f_1(x,t,u,v)<0, \sup_{t\in\mathbb R,x\in\mathbb R,\atop u\in \mathbb 
		R,v\geq 
		N}f_2(x,t,u,v)<0,\]
	\[\sup_{t\in\mathbb R,x\in\mathbb R, \atop u\geq 0,v\geq0}\frac{\partial 
		{f_1}}{\partial u}(x,t,u,v)<0,
	\sup_{t\in\mathbb R,x\in\mathbb R, \atop u\geq 0,v\geq0}\frac{\partial 
		{f_2}}{\partial v}(x,t,u,v)<0.\]
	$(\textbf{H3})$\label{h3} $f_i$ and $Df_i$ are uniformly almost periodic in 
	$t\in \mathbb 
	R$ with $x\in\mathbb R$ and
	$(u,v)$ in bounded sets.\\
	$(\textbf{H4})$ \label{h4}Let 
	\begin{equation}
		F(x,t,u,v)=\left(
		\begin{array}{ccc}
			f_1(x,t,u,v)\\ f_2(x,t,u,v)
		\end{array}
		\right).
	\end{equation}
	For any given sequences $\{x_n\}\subset \mathbb R$ and 
	$\{G_n\}\subset \textit{H}(F)$, there exist subsequences $\{ 
	x_{n_k}\}\subset 
	\{x_n\}$
	and $\{ G_{n_k}\}\subset \{G_n\} $ such that 
	$\lim\limits_{k\rightarrow\infty} G_{n_k}(x+  x_{n_k},t,u,v)$ exists
	for $t\in \mathbb 
	R$  uniformly and $(x,u,v)$ in bounded sets.}

In this paper, we take 
\begin{equation}\label{fi}
	\begin{array}{c}
		f_1(x,t,U,V)=a_1(x,t)(N_1-U)V-d_1(x,t)U,\\
		f_2(x,t,U,V)=a_2(x,t)(N_2-V)U-d_2(x,t)V.
	\end{array}
\end{equation}	
Define matrix function $A(x,t)$ by
\begin{equation}\label{A(x,t)}
	A(x,t):=\left(
	\begin{array}{ccc}
		\frac{\partial 
			f_1(x,t,U,0)}{\partial U}&\frac{\partial f_1(x,t,0,V)}{\partial 
			V}\\\frac{\partial 
			f_2(x,t,U,0)}{\partial U}& 
		\frac{\partial f_2(x,t,0,V)}{\partial V}
	\end{array}
	\right)=\left(\begin{array}{ccc}
		-d_1(x,t)&a_1(x,t) N_1\\a_2(x,t)N_2& 
		-d_2(x,t)
	\end{array}\right).
\end{equation}

Moreover, we assume that $A(x,t)$ satisfies\\
\textit{$(\textbf{H5})$ There exists some $L^*>0$ such that 
	$\inf\limits_{\tilde x\in\mathbb R,L\geq L^*}\lambda(A(\cdot+\tilde 
	x,\cdot),L)>0$.}\\
Where $\lambda(A,L)$ is the principal Lyapunov exponent and $L^*$ is a constant 
dependent
on $a_i(x,t), N_i, d_i(x,t)$ for $i=1,2$, which will be 
explicitly explained in section \ref{basic}.
\subsection{Main results}
\noindent
\\Next we will present our main results for problem (\ref{system-1}). In 
Section \ref{result1}, we will prove that $h^\prime(t)>0$ 
and $g^\prime(t)<0$ in $t\in (0, +\infty)$.	
Therefore, we denote $g_\infty:=\lim\limits_{t\rightarrow\infty}g(t)$, and 
$h_\infty:=\lim\limits_{t\rightarrow\infty}h(t).$ 
Further, we can obtain that $g_\infty\in[-\infty,0)$ and 
$h_\infty\in(0,\infty].$

\begin{theorem}[Existence and uniqueness]\label{existence}
	Assuming any given initial functions 
	$(U_{0},V_{0})$ satisfy (\ref{system-2}). For any 
	$\alpha\in(0,\alpha_0)$, there exists a time $T$ such that the system 
	(\ref{system-1}) admits  a 
	unique global solution $(U,V;g,h)\in 
	(C^{2+\alpha,1+\frac{\alpha}{2}}([g(t),h(t)]\times (0,T]))^2\times 
	(C^{1+\alpha/2}((0,T]))^{2}$, where $T$ is dependent on $\alpha, h_0, 
	||U_0||_{C^2([-h_0,h_0])}$ and $||V_0||_{C^2([-h_0,h_0])}$.
\end{theorem}
\begin{remark}
	Actually, the solution for system (\ref{system-1}) uniquely exists for all 
	$t\in(0,\infty)$ (See Theorem \ref{lemma-existence2}).
\end{remark}

In order to investigate the asymptotic dynamics of system (\ref{system-1}), we 
first introduce the following system,
\begin{equation}\label{system-10}
	\left\{\begin{array}{ll}{U_{t}=D_{1} U_{x x}+a_1(x,t)(N_1-U)V-d_1(x,t)U,} & 
		{-\infty<x<\infty, 
			\enspace t>0,} \\ {V_{t}=D_{2} V_{x 
				x}+a_2(x,t)(N_2-V)U-d_2(x,t)V,} & {-\infty<x<\infty, 
			\enspace 
			t>0.} 
	\end{array}\right.
\end{equation}
Then we get
\begin{theorem}[Spreading-vanishing dichotomy]\label{dichotomy}
	Supposing $(\textbf{H1})-(\textbf{H5})$ hold and the initial 
	functions 
	$(U_0,V_0)$ 
	satisfy $(\ref{system-2})$. Let $(U,V;U_0,V_0,h_0)$ be the solution of 
	$(\ref{system-1})$, for such $L^*$ in (\textbf{H5}), the following 
	spreading-vanishing dichotomy regimes 
	hold:
	\\Either\\(1) {\rm Vanishing}: $h_\infty-g_\infty\leq2L^*$ and 
	$$\lim\limits_{t\rightarrow 
		+\infty}U(x,t;U_0,V_0,h_0)=0, 
	\lim\limits_{t\rightarrow 
		+\infty}V(x,t;U_0,V_0,h_0)=0$$ uniformly in $x\in [g_\infty, h_\infty];$
	\\or\\(2) {\rm Spreading}: $h_\infty-g_\infty=\infty$ and 
	$$\lim\limits_{t\rightarrow 
		+\infty}U(x,t;U_0,V_0,h_0)-U^*(x,t)=0,
	\lim\limits_{t\rightarrow +\infty}V(x,t;U_0,V_0,h_0)-V^*(x,t)=0$$ locally 
	uniformly 
	for $x$ in 
	$\mathbb{R}$, where $(U^*(x,t),V^*(x,t))$ is the unique positive time 
	almost 
	periodic 
	solution of 
	(\ref{system-10}).
\end{theorem}

\begin{theorem}[Spreading-vanishing threshold]\label{threshold}
	Supposing $(\textbf{H1})-(\textbf{H5})$ hold. For any given $h(0), 
	g(0)$ and the initial 
	functions 
	$(U_0,V_0)$ 
	satisfying $(\ref{system-2})$. Let $(U,V;U_0,V_0,g,h)$ be the solution of 
	$(\ref{system-1})$, for such $L^*$ in (\textbf{H5}), the followings hold. 
	\\(1) If $\lambda(0)>0,$ then $h(0)-g(0)\geq 2L^*$, 
	further,
	$h_\infty-g_\infty=\infty$, thus, the 
	spreading
	occurs;
	\\(2) If $\lambda(0)<0$, then there exists a constant
	$\mu^*\geq0$ such that 
	the 
	spreading 
	occurs when $\mu>\mu^*$ and vanishing occurs when $0<\mu\leq \mu^*$.	
\end{theorem}

\begin{remark}
	{\rm The explicit explanation for principal Lyapunov exponent $\lambda(t)$ 
		can refer to section 
		\ref{basic}. The above theorem gives the sufficient conditions about 
		the 
		spreading 
		and 
		vanishing of the disease.	The threshold constant $L^*$ determines the 
		uniform persistence or extinction of WNv 
		by influencing the sign of the principal Lyapunov exponent (See section 
		\ref{result2}). }
\end{remark}

\begin{remark}
	In his paper, we always suppose 
	that 
	(\textbf{\textit{H5}}) holds. Accordng to Theorem \ref{threshold} and Lemma 
	\ref{propA}, considering the meaning 
	of biology, we can explain this 
	assumption in the sence that the living habitat at remote distance is in 
	high-risk of infection by the disease.  
\end{remark}

\section{Existence and Uniqueness}\label{result1}
\noindent
In this section, we will show the existence and uniqueness of the 
global solution for system 
(\ref{system-1}). Since the system
(\ref{system-1}) can be regarded as a special case of the system 
(\ref{system-3}). We only need to give a explicit proof for system 
(\ref{system-3}).
Although 
there are similar results about 
the solution for epidemic models with constant 
coefficients, the proofs of the (\ref{system-3}) in heterogeneous 
environment can not easily obtained by analogy. Therefore, we provide a 
detailed 
proof according to 
Theorem 1.1 (\cite{wang2019existence}).
\begin{theorem}\label{lemma-existence1}
	Assume that $(\textbf{H1})-(\textbf{H4})$ hold. For any $\alpha\in 
	(0,\alpha_0)$ and any given 
	$(u_{0},v_{0})$ 
	satisfying (\ref{system-2}), there exists $T>0$ 
	such that the system (\ref{system-3}) admits  a 
	unique solution $(u,v,g,h)\in 
	({C^{1+\alpha,(1+\alpha)/2}}({D_{T}}))^{2}\times 
	(C^{1+\alpha/2}([0,T]))^{2}$,
	where
	$D_{T}=\{(x,t)\in \mathbb{R}^{2}\mid x\in[g(t),h(t)],t\in[0,T]\}$,
	and	$T$ is only dependent on 
	$\alpha, h_0, ||u_0||_{C^2([-h_0,h_0])}$ and $||v_0||_{C^2([-h_0,h_0])}$.
\end{theorem}
\begin{proof}
	We divide this proof into two steps.\\
	\textbf{Step 1} The local existence of the solution for 
	problem 
	(\ref{system-3}).
	
	Let 
	\begin{equation}\label{equality-1}
		\begin{aligned}
			&y=\frac{2x}{h(t)-g(t)}-\frac{h(t)+g(t)}{h(t)-g(t)}, 
			\\&m(y,t)=u(x,t),
			n(y,t)=v(x,t),
			\\& \tilde f_1(y,t,m,n)=f_1(x,t,u,v),
			\\& \tilde f_2(y,t,m,n)=f_2(x,t,u,v),
		\end{aligned}
	\end{equation}
	then direct calculation gives
	\begin{equation}\label{equality-101}
		\begin{aligned}
			&\frac{\partial y}{\partial 
				x}=\frac{2}{h(t)-g(t)}:=\sqrt{A(y,g(t),h(t))},  
			\\&\frac{\partial y}{\partial 
				t}=-\frac{y({h^\prime(t)}-{g^\prime(t)})+({h^\prime(t)}
				+{g^\prime(t)})}{h(t)-g(t)}
			\\&\enspace\enspace:=B(y,g(t),g^\prime(t),h(t),{h^\prime(t)}),
		\end{aligned}
	\end{equation}
	and $(m,n)$ satisfy the following system,
	\begin{equation}\label{system-4}
		\left\{\begin{array}{ll}{m_{t}-D_{1}Am_{yy}+Bm_{y}=\tilde 
				f_1(y,t,m,n),} 
			& {y\in (-1,1),0<t\leq T,}\\{n_{t}-D_{2}An_{yy}+Bn_{y}=\tilde 
				f_2(y,t,m,n),} 
			& {y\in (-1,1),0<t\leq T,} \\ {m(\pm 1,t)=0 , n(\pm1,t)=0,} & 
			{0<t\leq 
				T, } \\ 
			{m(y,0)=u_{0}( 
				h_0 
				y),n(y,0)=v_{0}(h_0 y),} & {y\in[-1,1].} \end{array}\right.
	\end{equation}
	Meanwhile, $h(t)$ and $g(t)$ satisfy 
	\begin{equation}\label{system-5}
		\left\{\begin{array}{ll}{h(0)=h_0,h^{\prime}(t)=-\mu 
				\frac{2}{h(t)-g(t)}m_y(1,t),}&{0<t\leq 
				T,}\\{g(0)=-h_0,g^{\prime}(t)=-\mu 
				\frac{2}{h(t)-g(t)}m_y(-1,t),}&{0<t\leq T.}
		\end{array}\right.
	\end{equation}	
	Next, we will show the existence of 
	solution for 
	(\ref{system-4}) with (\ref{system-5}).
	
	Let $$h^*=-\mu u^{\prime}_0(h_0), g^*=-\mu u^{\prime}_0(-h_0),$$ 
	$$T_0=\min\left\{1,\frac{h_0}{2(2+h^*)},\frac{h_0}{2(2-g^*)}\right\},$$
	\begin{equation}\label{ga}
		\Gamma=\left\{h_0,h^{*},g^{*},||u_0||_{C^2([-h_0,h_0])},
		||v_0||_{C^2
			([-h_0,h_0])}\right\},
	\end{equation} 
	\[\Theta_T=\{(g,h)\in (C^1([0,T]))^2
	\mid h(0)=h_0, g(0)=-h_0, 
	h^{\prime}(0)=h^*,g^{\prime}(0)=g^*,\] 
	\[||h^{\prime}-h^*||_{L^\infty}\leq 
	1 ,||g^{\prime}-g^*||_{L^\infty}\leq 
	1 \},\] then $h^*>0,g^*<0$ and $\Theta_T$ is a bounded closed convex 
	subset of 
	$(C^{1}([0,T_0]))^2$ for 
	any $0<T\leq 
	T_0$. 
	
	Let 
	\[\Theta^*_{T_0}=\{(g,h)\in 
	(C^1([0,T_0]))^2\mid 
	h(0)=h_0,h^{\prime}(0)=h^*,g^{\prime}(0)=g^*,||h^{\prime}
	-h^*||_{L^\infty}\leq
	2,\]
	\[||g^{\prime}-g^*||_{L^\infty}\leq2 \}.\]
	For any given $(g,h)\in \Theta_T$, we can extend $h$ and $g$ such that 
	$(g,h)\in 
	\Theta^*_{T_0}$. 
	Hence, if $(g,h)\in \Theta_T$, then $(g,h)\in \Theta^*_{T_0}$.  
	And	$h(t)$ and $g(t)$ satisfy
	\begin{equation}
		\begin{aligned}
			&|h(t)-h_0|\leq T_0||h^{\prime}||_\infty\leq T_0(2+h^*)\leq 
			\frac{h_0}{2},
			\\&|g(t)-(-h_0)|\leq T_0||g^{\prime}||_\infty\leq 
			T_0(2+g^*)\leq T_0(2-g^*)\leq
			\frac{h_0}{2}
		\end{aligned}
	\end{equation}
	for 
	any $t\in [0,T_0]$, 
	then 
	$h(t)\in[\frac{h_0}{2},\frac{3 h_0}{2}]$ and 
	$g(t)\in[-\frac{3h_0}{2},-\frac{ 
		h_0}{2}]$ in $[0,T_0]$. Hence, 
	the transformations (\ref{equality-1})  and (\ref{equality-101}) are well 
	defined for 
	$t\in [0,T_0]$. Applying the standard parabolic equation theory 
	(\cite{1968linear}), 
	there 
	exists a $T_*\in (0,T_0]$ such that 
	there is 
	a unique solution $( {\overline m}(y,t),\overline n(y,t))\in 
	(C^{1+\alpha,\frac{1+\alpha}{2}}(\Delta_{T_*}))^2$ for problem 
	(\ref{system-4}) 
	with 
	$T_*$ 
	dependent on $\Ga, \alpha,
	||u_0||_\infty$ and $|||v_0||_\infty$.
	And there exists a positive constant $C_1(\Ga,\alpha,T_*,T^{-1}_*)$ 
	such that
	\[
	||\overline 
	m||_{C^{1+\alpha,\frac{1+\alpha}{2}}(\Delta_{T_*})}+
	+||\overline 
	n||_{C^{1+\alpha,\frac{1+\alpha}{2}}(\Delta_{T_*})}\leq
	C_1(\Ga,\alpha,T_*,T^{-1}_*),\]
	where $\Delta_{T_*}=[-1,1]\times [0,T_*]$.
	In view of the choice of $\Gamma$ in (\ref{ga}), $T_*$ is only dependent on 
	$\Gamma$ and $\alpha$. Hence, \[||\overline 
	m||_{C^{1+\alpha,\frac{1+\alpha}{2}}(\Delta_{T_*})}+
	+||\overline 
	n||_{C^{1+\alpha,\frac{1+\alpha}{2}}(\Delta_{T_*})}\leq
	C_1,\]
	with $C_1$ dependent on $\Gamma,\alpha$.
	Moreover, for $0<T\leq T_*$, it follows
	\begin{equation}\label{inequality-1}
		||\overline 
		m||_{C^{1+\alpha,\frac{1+\alpha}{2}}(\Delta_{T})}+
		||\overline 
		n||_{C^{1+\alpha,\frac{1+\alpha}{2}}(\Delta_{T})}\leq
		C_1.
	\end{equation}
	Since $(\overline m(y,0)$ and $ \overline n(y,0)) $ are more than but not 
	identically
	equal to 0 for $y\in 
	[-1,1]$, $\tilde f_i(y,t,0,0)\geq 0$, $a_iN_i\geq 0$ on $[-1,1]\times[0,T]$
	and $\tilde 
	f_i(y,t,m,n)$ satisfies $\textit{(\textbf{H1})}$ for $i=1,2$,
	by the maximum principle (resp. Positivity Lemma, \cite{wang2021}), 
	then $(\overline 
	m, \overline 
	n)>0$ for $(y,t)\in 
	(-1,1)\times (0,T]$. 
	
	Consider that the solution 
	$(\overline m,\overline n)$ depends continuously on the initial data 
	$(g,h)\in \Theta_T$. Let
	\begin{equation}
		\overline h(t)=h_0-\mu \int^{t}_{0}\dfrac{2}{h(s)-g(s)}\overline 
		m_y(1,s)ds,\overline 
		g(t)=-h_0-\mu \int^{t}_{0}\dfrac{2}{h(s)-g(s)}\overline m_y(-1,s)ds,
	\end{equation}
	for $t\in [0,T]$,
	then $(\overline g,\overline h)$ depend on $(g,h)\in \Theta_T$ and
	\[\overline h(0)=h_0,\overline h^{\prime}(0)=h^*, \bar h^{\prime}(t)>0,
	\overline g(0)=-h_0,\overline g^{\prime}(0)=g^*, \bar g^{\prime}(t)<0.
	\]
	Moreover, it follows
	\begin{equation}\label{inequality-2}
		\begin{aligned}				
			&	\overline h^{\prime}(t) \in C^{\frac{\alpha}{2}}([0,T]), 
			||\overline 
			h^{\prime}(t)||_{C^{\frac{\alpha}{2}}([0,T])}\leq 
			C_2,
			\\&\overline g^{\prime}(t) \in 
			C^{\frac{\alpha}{2}}([0,T]), 
			||\overline 
			g^{\prime}(t)||_{C^{\frac{\alpha}{2}}([0,T])}\leq 
			C_2,
		\end{aligned}
	\end{equation}
	for some $C_2$ dependent on $\Gamma, \alpha$.
	
	Define $\mathcal{F}:\mathcal{D}_1\times \mathcal{D}_2\times 
	\Theta_T\longrightarrow 
	C(\Delta _T)\times 
	C(\Delta _T)\times (C^1({[0,T]}))^2$ by
	\[\mathcal{F}(m,n,g,h)=(\overline m, \overline n, \overline g, 
	\overline h),\] 
	where
	\[\mathcal{D}_1=\{m\in C(\Delta 
	_T)|m(y,0)=u_0(h_0y),||m-u_0||_{C(\Delta 
		_ T)}\leq1\},\]
	\[\mathcal{D}_2=\{n\in C(\Delta_ 
	T)|n(y,0)=v_0(h_0y),||n-v_0||_{C(\Delta_ 
		T)}\leq 
	1\}.\]
	It is easily to see that $\mathcal{F}(m,n,g,h)=( m,  n, g, h)$ if and 
	only if 
	$(m,n,g,h)$ is the solution of (\ref{system-4}) with (\ref{system-5}).
	
	Combining (\ref{inequality-1}) and (\ref{inequality-2}), it follows
	\begin{equation}
		\begin{aligned}
			&||\overline m-u_0||_{C(\Delta _T)}+||\overline 
			n-v_0||_{C(\Delta_ 
				T)}\\&\leq 
			||\overline 
			m-u_0||_{C^{{\frac{1+\alpha}{2}},0}(\Delta 
				_T)}T^{{\frac{1+\alpha}{2}}}+||\overline 
			n-v_0||_{C^{{\frac{1+\alpha}{2}},0}(\Delta 
				_T)}T^{{\frac{1+\alpha}{2}}}
			\\&\leq 
			C_1 T^{\frac{1+\alpha}{2}}
		\end{aligned}
	\end{equation}
	and
	\[||\overline h^{\prime}-h^{*}||_{C([0,T])}\leq ||\overline 
	h^{\prime}||_{C^{\frac{\alpha}{2}}([0,T])}T^{\frac{\alpha}{2}}\leq 
	C_2 T^{\frac{\alpha}{2}},\]
	\[||\overline g^{\prime}-g^{*}||_{C([0,T])}\leq ||\overline 
	g^{\prime}||_{C^{\frac{\alpha}{2}}([0,T])}T^{\frac{\alpha}{2}}\leq 
	C_2 T^{\frac{\alpha}{2}}.\]
	Therefore, if we take 
	$T=\min\left\{1,\frac{h_0}{2(2+h^*)}, \frac{h_0}{2(2-g^*)}, 
	C^{-\frac{2}{1+\alpha}}_1,
	C^{-\frac{2}{\alpha}}_2\right\}
	$, then $\mathcal F$ maps 
	$\mathcal{D}_1\times \mathcal{D}_2\times\Theta_T$ into itself.
	Further, we can get that $\mathcal F$ is compact.  Applying the 
	Schauder fixed point theorem 
	to $\mathcal F$, there exists 
	a solution $(m,n,g,h)\in \mathcal{D}_1\times \mathcal{D}_2\times 
	\Theta_T$. Applying the Schauder estimates, 
	$(m,n,g,h)\in (C^{1+\alpha,\frac{1+\alpha}{2}}([-1,1]\times 
	[0,T]))^2\times C^{1+\frac{\alpha}{2}}([0,T])^2$ 
	of system
	(\ref{system-4}) with (\ref{system-5}). Hence, the problem 
	(\ref{system-3}) has a solution $(u,v,g,h)\in 
	(C^{1+\alpha, \frac{1+\alpha}{2}}([g(t),h(t)]\times [0,T]))^2\times 
	C^{1+\frac{\alpha}{2}}([0,T])^2$.
	
	\textbf{\textbf{Step 2}} The uniqueness of the solution for 
	problem 
	(\ref{system-3}).
	
	Assume that $(u_i,v_i,g,h) (i=1,2) \in \mathcal{D}_1\times 
	\mathcal{D}_2\times\Theta_T$ are the two solutions of (\ref{system-3}) 
	for 
	$0<T\ll 1$.
	Applying the strong maximum 
	principle to $u_i$, we can get that $u_i(x,t)>0$ for $x\in (g(t),h(t))$ 
	and 
	$0<t<T$. 
	In 
	view of 
	$u_i(t,h(t))=0,u_i(t,g(t))=0$, it follows  
	${u_i}_x(t,h(t))<0,{u_i}_x(t,g(t))>0$ 
	for 
	$i=1,2$, which 
	implies 
	$h^{\prime}(t)>0,g^{\prime}(t)<0$ for $t\in (0,T)$, then we can
	suppose that\begin{equation}\label{equality-9} 
		h_0\leq h(t)\leq h_0+1, -h_0-1\leq g(t)\leq -h_0
	\end{equation} 
	for $t\in[0,T]$,
	and $$u_i\leq ||u_0||_\infty +1, 
	v_i\leq ||v_0||_\infty +1$$ in $[g(t),h(t)]\times [0,T]$ for $i=1,2$.
	
	As in transformations (\ref{equality-1}), take $$m_i(y,t)=u_i({x},t), 
	n_i(y,t)=v_i({x},t)$$ 
	for 
	$i=1,2$, then 
	$(y,t)\in [-1,1]\times [0,T]$.
	
	Let $$m=m_1-m_2, n=n_1-n_2, h=h_1-h_2,g=g_1-g_2,$$ direct calculation 
	gives 
	the following system,
	\begin{equation}\label{system-6}
		\left\{\begin{array}{ll}{m_{t}-D_{1}A_1(y,t)m_{yy}+B_1(y,t)m_{y}
				-a_1(y,t)m-\tilde 
				a_1(y,t)n}\\{=D_1(A_1-A_2){m_2}_{yy}+(B_2-B_1){m_2}_y+
				b_1(y,t)\frac{y(h-g)+(h+g)}{2}}, & {y\in
				(-1,1),0<t\leq T},
			\\{n_{t}-D_{2}A_1(y,t)n_{yy}+B_1(y,t)n_{y}
				-a_2(y,t)m-\tilde 
				a_2(y,t)n}\\{=D_2(A_1-A_2){n_2}_{yy}+(B_2-B_1){n_2}_y
				+b_2(y,t)\frac{y(h-g)+(h+g)}{2}},
			&
			{y\in (-1,1),0<t\leq T},
			\\ {m(\pm 1,t)=0 , n(\pm1,t)=0}, 
			&{0<t\leq 
				T},  
			\\ {m(y,0)=n(y,0)=0},
			&{y\in(-1,1)}, 
		\end{array}\right.
	\end{equation}
	with 
	\begin{equation} \label{equality-5}
		\begin{aligned}
			h^\prime(t)=\mu(\frac{2}{h_2(t)-g_2(t)}{m_2}_y(1,t)-
			\frac{2}{h_1(t)-g_1(t)}{m_1}_y(1,t)),
			&\\g^\prime 
			(t)=\mu(\frac{2}{h_2(t)-g_2(t)}{m_2}_y(-1,t)-
			\frac{2}{h_1(t)-g_1(t)}{m_1}_y(-1,t)),
		\end{aligned}
	\end{equation} 
	for $0<t\leq 
	T,h(0)=0,g(0)=0,$ $i=1,2,$
	where  
	\[ A_i(y,t)=\frac{4}{(h_i(t)-g_i(t))^2},
	\]					
	\[B_i(y,t)=-\frac{y({h^\prime_i(t)}-{g^\prime_i(t)})+({h^\prime_i(t)}
		+{g^\prime_i(t)})}{h_i(t)-g_i(t)},\]
	\[b_i(y,t)=\int_{0}^{1} {f_i}_x(t,H(h_1,h_2,g_1,g_2,s),m_2,n_2)ds,\]
	\[\tilde a_i(y,t)=\int_{0}^{1} {f_i}_n(t,\frac{y(h_1 
		-g_1)+(h_1+g_1)}{2},m_2,n_2+s(m_1-m_2))ds,\]
	\[ a_i(y,t)=\int_{0}^{1} {f_i}_m(t,\frac{y(h_1 
		-g_1)+(h_1+g_1)}{2},m_2+s(m_1-m_2),n_1))ds,
	\]
	\[H(h_1,h_2,g_1,g_2)=\frac{y(h_2+s(h_1 -h_2)-(g_2+s(g_1 
		-g_2))+(h_2+s(h_1-h_2)+(g_2+s(g_1-g_2)))}{2}.\]
	In view of $(\textbf{\textit{H1}})-(\textbf{\textit{H5}})$, we can get 
	$a_i,\tilde a_i, 
	b_i\in 
	L^\infty(\Delta_T)$ for $i=1,2$ with $||a_i(y,t)||_{L^\infty},||\tilde 
	a_i(y,t)||_{L^\infty}$ and 
	$||b_i(y,t)||_{L^\infty}$ dependent on 
	$h_0,||u_0||_{L^\infty}$ and $||v_0||_{L^\infty}$.
	In view of $(\ref{inequality-1})-(\ref{equality-9})$, applying $L^p$ 
	theory for parabolic equations and Sobolev imbedding theorem to system 
	$(\ref{system-6})$, 
	there are positive constants $C_3,C_4$ and 
	$C_5$ which depend on $\Gamma, \alpha$ such that
	\begin{equation}\label{mn}
		\begin{aligned}
			&||m||_{C^{1+\alpha,\frac{1+\alpha}{2}}(\Delta 
				_T)}+||n||_{C^{1+\alpha,\frac{1+\alpha}{2}}(\Delta_T)}\leq 
			C_3 (D_1 
			||((h_1-g_1)^{-2}-(h_2-g_2)^{-2}){m_2}_{yy}||_{C(\Delta_ 
				T)}\\&+|b_1 
			\dfrac{y(h-g)+(h+g)}{2}||_{C(\Delta_ T)} 
			+||(\dfrac{y(h^\prime_1-g^\prime_1)
				+(h^\prime_1+g^\prime_1)}{h_1(t)-g_1(t)}
			-\dfrac{y(h^\prime_2-g^\prime_2)+(h^\prime_2+
				g^\prime_2)}{h_2(t)-g_2(t)}){m_2}_y||_{C(\Delta
				_T)})
			\\&+C_3 
			(D_2 
			||((h_1-g_1)^{-2}-(h_2-g_2)^{-2}){n_2}_{yy}||_{C(\Delta_ 
				T)}+|b_2 
			\dfrac{y(h-g)+(h+g)}{2}||_{C(\Delta_ T)}
			\\&+||(\dfrac{y(h^\prime_1-g^\prime_1)+(h^\prime_1
				+g^\prime_1)}{h_1(t)-g_1(t)}-
			\dfrac{y(h^\prime_2-g^\prime_2)+(h^\prime_2+
				g^\prime_2)}{h_2(t)-g_2(t)}){n_2}_y||_{C(\Delta
				_T)})
			\\&\leq 
			C_4(||h||_{C^1([0,T]))}+||g||_{C^1([0,T]))}
			+||h-g||_{C^1([0,T]))}
			+||h+g||_{C^1([0,T]))})
			\\&\leq C_5(||h||_{C^1([0,T]))}+||g||_{C^1([0,T]))}).
		\end{aligned}
	\end{equation}

	Applying the proofs of $(5.4.3)$ and Theorem 
	$5.5.4$ (\cite {wangbook}) to
	$m_y(y,t)$ and $n_y(y,t)$, without needing to expand $m$ and $n$ to a 
	larger domain, we obtain that there exists a positive constant $\tilde 
	C_1$ 
	independent of 
	$T^{-1}$ 
	such 
	that
	$$
	\begin{array}{c}
		{[m]_{C^{\alpha,\frac{\alpha}{2}}(\Delta_ 
				T)}+[m_y]_{C^{\alpha,\frac{\alpha}{2}}(\Delta_ T)}\leq 
			\tilde	C_1||m||_{C^{1+\alpha,\frac{1+\alpha}{2}}(\Delta_T)},}\\
		{[n]_{C^{\alpha,\frac{\alpha}{2}}(\Delta_ 
				T)}+[n_y]_{C^{\alpha,\frac{\alpha}{2}}(\Delta_ T)}\leq 
			\tilde C_1||n||_{C^{1+\alpha,\frac{1+\alpha}{2}}(\Delta_T)},}
	\end{array}
	$$
	where [ $\cdot $ ] is the H$\ddot{\rm o}$lder seminorm.
	Therefore, according to (\ref{mn}) and the above inequalities, it follows  
	that
	\begin{equation}\label{equality-8}
		\begin{aligned}
			&	{	[m_y]_{C^{\alpha,\frac{\alpha}{2}}(\Delta _T)}\leq 
				\tilde	C_1 
				C_5(||h||_{C^1([0,T]))}+||g||_{C^1([0,T]))}),} \\
			&	{[n_y]_{C^{\alpha,\frac{\alpha}{2}}(\Delta _T)}\leq 
				\tilde	C_1 
				C_5(||h||_{C^1([0,T]))}+||g||_{C^1([0,T]))}).}
		\end{aligned}
	\end{equation}
	Combining (\ref{equality-5}) and (\ref{equality-8}), there is $C_6$ 
	dependent on $\Gamma, \alpha$ such that
	\begin{equation}
		\begin{aligned}
			&[h^\prime]_{C_{\frac{\alpha}{2}([0,T])}}\leq 
			\mu[\frac{2}{h_1-g_1}m_y(1,t)]_{C_{\frac{\alpha}{2}([0,T])}}
			+\mu[(\frac{2}{h_1-g_1}-\frac{2}{h_2-g_2})
			{m_2}_y(1,t)]_{C_{\frac{\alpha}{2}([0,T])}}
			\\&\quad\quad\quad\quad\enspace \leq 
			C_6(||h||_{C^1([0,T])}+||g||_{C^1([0,T])})
		\end{aligned}
	\end{equation}
	and
	\begin{equation}
		\begin{aligned}
			&[g^\prime]_{C_{\frac{\alpha}{2}([0,T])}} \leq 
			\mu[\frac{2}{h_1-g_1}m_y(-1,t)]_{C_{\frac{\alpha}{2}([0,T])}}
			+\mu[(\frac{2}{h_1-g_1}-\frac{2}{h_2-g_2})
			{m_2}_y(-1,t)]_{C_{\frac{\alpha}{2}([0,T])}}
			\\&\quad\quad\quad\qquad\leq 
			C_6(||h||_{C^1([0,T])}++||g||_{C^1([0,T])}).
		\end{aligned}
	\end{equation}
	Since $h(0)=h^{\prime}(0)=0$ and $g(0)=g^{\prime}(0)=0$, then 
	$$
	\begin{array}{c}
		||h-h(0)||_{C^1([0,T])}\leq 
		2||h^{\prime}-h^{\prime}(0)||_{C^{\frac{\alpha}{2}}([0,T])}
		T^{\frac{\alpha}{2}}
		\leq \tilde C_6||h||_{C^1([0,T])}T^{\frac{\alpha}{2}},
		\\
		||g-g(0)||_{C^1([0,T])}\leq 
		2||g^{\prime}-g^\prime(0)||_{C^{\frac{\alpha}{2}}([0,T])}
		T^{\frac{\alpha}{2}}
		\leq \tilde C_6||g||_{C^1([0,T])}T^{\frac{\alpha}{2}}.
	\end{array}
	$$
	Therefore, if $T$ is small enough, then $h=0$ and $g=0$, which 
	implies $m=0$ 
	and $n=0$.
	Thus, the local existence and uniqueness of the solution have been 
	proved. 
\end{proof}
\begin{proof}[Proof of Theorem \ref{existence}] 
	Let $f_i$ be defined by (\ref{fi}) for $i=1,2$. In view that 
	$a_1(x,t),a_2(x,t),d_1(x,t),d_2(x,t)\in 
	C^{2+\alpha,1+\frac{\alpha}{2}}(\mathbb R\times [0,\infty))$  for 
	any $\alpha\in(0,\alpha_0)$ and have positive upper and lower bound, it 
	follows that
	$f_i(x,\cdot,U,V)\in C^{1+\frac{\alpha}{2}}([0,T])(i=1,2)$ for the  
	$T$ in Theorem \ref{lemma-existence1}.  
	
	Make the transformations as (\ref{equality-1}), combining 
	(\ref{inequality-1}) with (\ref{inequality-2}),
	then, it can be obtained that
	$$\tilde 
	f_i(y,t):=f_i(\dfrac{y(h(t)-g(t))+(h(t)+g(t))}{2},t,m(y,t),n(y,t))\in 
	C^{\alpha,\frac{\alpha}{2}}([-1,1]\times [0,T]).$$
	Using the Schauder theory for parabolic equations to system 
	(\ref{system-4}) and (\ref{system-5}), we can get that 
	$(m,n,g,h)\in 
	(C^{2+\alpha,1+\frac{\alpha}{2}}([-1,1]\times (0,T]))^2\times 
	C^{1+\frac{\alpha}{2}}((0,T])^2.$
	Since the system 
	(\ref{system-1}) can be regarded as the 
	special case of (\ref{system-3}), and satisfies all of the assumptions in 
	Theorem \ref{lemma-existence1}, thus, the system (\ref{system-1}) admits a  
	unique solution $(U,V;g,h)\in 
	(C^{2+\alpha,1+\frac{\alpha}{2}}([g(t),h(t)]\times (0,T]))^2\times 
	C^{1+\frac{\alpha}{2}}((0,T])^2$. Thus, the 
	local existence and uniqueness of solution for system (\ref{system-1}) are 
	proved.
\end{proof}	

For the convenience of later proof, we provide 
the following Comparison Principle in order to estimate the boundness of 
$U(x,t)$, 
$V(x,t)$ for system 
(\ref{system-1}) and 
the free boundaries $x=g(t)$, $x=h(t)$. The lemma is similar to Lemma 3.5 
in \cite{du2010spreading}.
\begin{lemma}[Comparison Principle]\label{lemma-comparison}\quad
	Assume that $T\in(0,+\infty)$, $\overline{h}(t), \overline{g}(t)\in 
	C^{1}([0,T])$, $\overline{U},\overline{V}\in 
	C(\overline{D_{T}^{*}})\bigcap C^{2,1}(D_{T}^{*})$ with $0<\overline 
	U\leq N_1, 0<\overline V\leq N_2$ and $(\overline 
	U,\overline V;\overline h,\overline g)$ satisfies
	\begin{equation}\label{2.4}
		\left\{\begin{array}{ll}{\overline{U}_{t}-D_{1} \overline{U}_{x x} \geq 
				a_1(x,t)(N_1-\overline U)\overline V-d_1(x,t)\overline U,} & 
			{\overline{g}(t)<x<\overline{h}(t), \enspace 0<t<T,} 
			\\ {\overline{V}_{t}-D_{2} \overline{V}_{x x} \geq 
				a_2(x,t)(N_2-\overline V)\overline U-d_2(x,t)\overline V,} & 
			{\overline{g}(t)<x<\overline{h}(t), \enspace 0<t<T,} 
			\\ {\overline{U}(x, t)\geq 0,\overline{V}(x, t)\geq 0,} 
			& {x=\overline{g}(t)\enspace \rm or\enspace \overline{h}(t), 
				\enspace 
				0<t<T,} \\  {\overline h(0)\geq h_0, \overline{h}^{\prime}(t) 
				\geq-\mu 
				\overline{U}_{x}(\overline{h}(t), t),} & {\enspace 0<t<T,}\\ 
			{\overline g(0)\leq -h_0, \overline{g}^{\prime}(t) \leq-\mu 
				\overline{U}_{x}(\overline{g}(t), t),} & {\enspace 0<t<T,}
			\\\overline{U}(x, 0)\geq {U_{0}(x), \enspace\overline{V}(x, 0)\geq 
				V_{0}(x),} & {-h_{0} \leq x \leq h_{0},}\end{array}\right.
	\end{equation}
	then the solution  $(U,V;g,h)$ of (\ref{system-1}) satisfies
	\begin{equation}\label{2.5}
		\begin{aligned}
			&\overline{ U}(x,t)\geq {U}(x,t), \enspace \overline{V}(x,t)\geq 
			{V}(x,t),
			\\ & \overline{h}(t)\geq {h}(t),\enspace {g}(t)\geq \overline{ 
				g}(t),\enspace 
			{\rm for}\enspace g(t)\leq x\leq h(t),t\in(0,T],
		\end{aligned}
	\end{equation}
	where $D_{T}^{*}=\{(x,t)\in\mathbb{R}^{2}\mid x\in(\overline 
	g(t),\overline{h}(t)),t\in(0,T]\}$.
\end{lemma}

\begin{remark}
	If $(\overline{U},\overline{V};\overline{g},\overline{h})$ satisfies 
	the conditions of Lemma \ref{lemma-comparison}, then it is called  the 
	upper 
	solution of (\ref{system-1}).
	The corresponding lower solution can be similarly defined by reversing 
	the above inequalities.	
\end{remark}	

In order to extend the local solution of (\ref{system-1}) to all 
$t\in(0,\infty)$, 
according  
to Lemma 2.2 
in 
\cite{du2010spreading} or Lemma 2.5 and  Lemma 2.6 in \cite{cheng2020}, we give 
the rough estimates about the
supper and lower bound of $U(x,t), V(x,t), g^\prime(t)$ and $h^\prime(t)$.
\begin{lemma}\label{bounded}
	Assume that $T\in(0,+\infty)$. Let $(U,V;g,h)$ be a solution of 
	(\ref{system-1}) for $t\in(0,T]$, then there exists a positive constant 
	$C>0$ independent 
	of $T$ such that
	\begin{equation}
		\begin{aligned}
			&	0<U(x,t)\leq  N_1,0<V(x,t)\leq  N_2,\enspace{\rm 
				for}\enspace g(t)< 
			x< h(t), 0< t\leq T,
			\\&0<h^{\prime}(t)\leq C, -C\leq g^{\prime}(t) <0,\enspace {\rm 
				for}\enspace 
			0<t\leq T.
		\end{aligned}
	\end{equation}
\end{lemma}
Now we turn to show the global existence of the solution for problem 
$(\ref{system-1})$.
\begin{theorem}\label{lemma-existence2}
	For any given initial data
	$(U_{0},V_{0})$ 
	satisfying (\ref{system-2}), the unique solution $(U,V;g,h)$ of system
	$(\ref{system-1})$ exists for all 
	$t\in(0,\infty)$.
\end{theorem}
\begin{proof}
	Now we aim to show that the solution for system (\ref{system-1}) can extend 
	to all 
	$t\in (0,\infty)$.
	If the maximal existence interval of the solution is $[0,T_{max})$, then we 
	will 
	show $T_{max}=+\infty$. On the contrary, assuming that $T_{max}<+\infty$. 
	According to Lemma \ref{bounded}, we can get that
	$
	U(x,t)\leq  N_1,  V(x,t) \leq  N_2 $ for $(x,t)$ in 
	$[g(t), 
	h(t)]\times \left[0, T_{max}\right) .
	$
	Moreover, for the above positive constant $C$ in Lemma \ref{bounded} 
	independent on 
	$T_{max}$, it holds that $0<h^\prime(t),-g^\prime(t)\leq  C$,
	follows $h_{0} 
	\leq h(t) \leq h_{0}+ C T_{max}$ 
	and $-h_{0}- CT_{max} \leq g(t) \leq -h_{0}$ 
	for $t\in [0, T_{max})$.
	
	As in transformation
	(\ref{equality-1}), take $m(y,t)=U(x,t), n(y,t)=V(x,t)$. For any given 
	$T<T_{max},$ applying 
	the $L^{p}$ theory to 
	(\ref{system-1}), there exists a positive constant $C_{1}\left(\Ga,  N_1,  
	N_2, T_{max}\right)$ independent of $T$ such that  
	$\|m\|_{W_{p}^{2,1}\left(\Delta_{T}\right)}+\|n\|_{W_{p}^{2,1}\left(\Delta_{T}\right)}
	\leq 
	\tilde C_{1}\left(\Ga,  N_1,  N_2, T_{max}\right)$, thus, $(m,n)
	\in (W_{p}^{2,1}\left(\Delta_{T_{max}}\right))^2$ for 
	$p>\frac{3}{1-\alpha}$ and
	$$
	\|m\|_{W_{p}^{2,1}\left(\Delta_{T_{max}}\right)}
	+\|m\|_{C^{1+\alpha,\frac{1+\alpha}{2}
		}\left(\Delta_{T_{max}}\right)}+
	\|n\|_{W_{p}^{2,1}\left(\Delta_{T_{max}}\right)}+
	\|n\|_{C^{1+\alpha,\frac{1+\alpha}{2}
		}\left(\Delta_{T_{max}}\right)} \leq \tilde C_{1}\left(\Ga,  
	N_1,  N_2,
	T_{max}\right).
	$$
	In view of (\ref{equality-5}), we can get $(h,g) \in 
	(C^{1+\frac{\alpha}{2}}\left(\left[0, 
	T_{max}\right]\right))^2$ and
	\begin{equation}\label{equation-18}
		\|h\|_{C^{1+\frac{\alpha}{2}}\left(\left[0, T_{max}\right]\right)} 
		\leq 
		C_{2}\left(\Ga,  N_1,  N_2, T_{max}\right),
		\|g\|_{C^{1+\frac{\alpha}{2}}\left(\left[0, T_{max}\right]\right)}
		\leq 
		\tilde C_{2}\left(\Ga,  N_1,  N_2, T_{max}\right).
	\end{equation}
	Applying the 
	Schauder theory 
	to
	(\ref{system-1}), we can  get that $$(m,n) \in (C^{
		2+\alpha,1+\frac{\alpha}{2}}\left([-1,1]\times\left(0, 
	T_{max}\right] \right))^2,$$ and  
	it holds that
	$$
	\|m\|_{C^{2+\alpha,1+\frac{\alpha}{2} }\left(\left[-1,1]\times[\varepsilon, 
		T_{max}\right] \right)}+\|n\|_{C^{2+\alpha,1+\frac{\alpha}{2} 
		}\left(\left[-1,1]\times[\varepsilon, 
		T_{max}\right] \right)} \leq \tilde C_{3}\left(\varepsilon, \Ga, N_1, 
	N_2, 
	T_{max}\right)
	$$
	for any small $0<\varepsilon \ll T_{max}$.
	Therefore, $(U,V)\in (C^{2+\alpha,1+\frac{\alpha}{2}}\left(\left [g(t), 
	h(t)]\times (0, 
	T_{max}\right]\right))^2$ and
	\begin{equation}\label{equation-19}
		\|U\|_{C^{2+\alpha,1+\frac{\alpha}{2} 
			}\left(\left[g(t),h(t)]\times[\varepsilon, 
			T_{max}\right] \right)}+\|V\|_{C^{2+\alpha,1+\frac{\alpha}{2} 
			}\left(\left[g(t),h(t)]\times[\varepsilon, 
			T_{max}\right] \right)}\leq \tilde C_{3}\left(\varepsilon, \Ga, 
		N_1, N_2, T_{max}\right).
	\end{equation}
	Thus, the system (\ref{system-1}) admits a solution $(U,V;g,h)$ on 
	$(0,T_{max}]$. 
	Take $\{T_{n}\} \subset\left(0, T_{max}\right)$ such that $T_{n} 
	\rightarrow 
	T_{max}$ as 
	$n\rightarrow \infty$.
	Let $T_{n}$ and $\left(U\left(x,T_n\right), V\left(x,T_n\right);g(T_n),
	h\left(T_{n}\right)\right)$ be the initial state. 
	By Theorem \ref{lemma-existence1}, there is a constant 
	$t_0$ small enough  dependent on 
	$$g\left(T_{n}\right), g^{\prime}\left(T_{n}\right),h\left(T_{n}\right), 
	h^{\prime}\left(T_{n}\right), 
	\left\|U\left(\cdot, T_{n}\right)\right\|_{C^2([g(T_n), 
		h\left(T_{n}\right)])}, \left\|V\left(\cdot, 
	T_{n}\right)\right\|_{C^2([g(T_n), 
		h\left(T_{n}\right)])}$$ such that problem (\ref{system-1}) admits a 
	unique 
	solution 
	$\left(U_{n}, V_n; g_n, h_{n}\right)$ for $ t\in \left[T_{n}, 
	T_{n}+t_0\right].$ 
	Considering the 
	uniqueness of the solution for (\ref{system-1}), it follows that the 
	solution $(U,V;g, 
	h)=\left(U_{n}, 
	V_n; g_n,
	h_{n}\right)$ for $T_{n} \leq$ 
	$t<\min 
	\left\{T_{n}+t_0, T_{max}\right\},$ which implies that  
	the solution $(U,V;g, 
	h)$ 
	for 
	(\ref{system-1}) can be extended to $\left[0, T_{n}+t_0\right)$. 
	In view of 
	(\ref{equation-18}) and (\ref{equation-19}),  $t_0$ can be taken
	independent 
	of
	$n$ such that $T_{n}+t_0>T_{max}$, which is contradict to the choice of 
	$T_{max}$. Thus, this theorem
	has been proved.
\end{proof}

\section{Principal Lyapunov exponent}\label{basic}
\noindent
In order to investigate the global dynamics for model (\ref{system-1}), 
considering the spatial heterogeneity and temporal almost periodicity, we first 
introduce the principal Lyapunov exponent 
and explore
several valuable properties, which will be frequently used in later studies. 

For any given $L>0$ and the uniformly almost periodic matrix function 
$A(x,t)$ defined by (\ref{A(x,t)}), 
consider the the following equation,
\begin{equation}\label{equation-2}
	\left\{\begin{array}{ll}{I_t=D(x,D)I+A(x,t)I,} & {-L<x<L,\enspace t>0}, 
		\\ 
		{I(-L,t)=I(L,t)=0,}& {t>0},\end{array}\right.
\end{equation}
where $-D(x,D)$ is a second-order strongly elliptic differential operator 
matrix of diagonal type with 
$D(x,D)=(D_i \partial_{ii})$ for $i=1,2$.

Let $X\hookrightarrow C^2([-L,L])\times 
C^2([-L,L])$ be 
the 
fractional power space (Chapter 1, ~\cite{henry1981}) with respect to the 
sectorial 
operator 
$-D(x,D)$ with homogeneous 
Cauchy boundary 
conditions, where $\mathcal D(-D(x,D) )=\{(u,v) \in$
$(C^{2}([-L,L]))^2\mid u(\pm L)=v(\pm L)=0 \}$.
By the standard semigroup theory (\cite{pazy2012}), for any $I_0\in X$, there 
exists a unique 
solution $I(t,\cdot;I_0,A)$ of 
(\ref{equation-2}) satisfying $I(0,\cdot;I_0,A)=I_0(\cdot)$.
\begin{definition}[ DEFINITION 
	4.3, Part \rm \uppercase\expandafter{\romannumeral2},\cite{shen1998}]
	We define the principal Lyapunov exponent $\lambda(A,L)$ of 
	(\ref{equation-2}) as 
	\[\lambda(A,L)=\lim\limits_{t\rightarrow+\infty}\sup\dfrac{\rm{ln}||\Phi(A,t)||_X}
	{t},\]
	where $\Phi(A,t)$ 
	satisfies $\Phi(A,t)I_0=I(t,\cdot;I_0,A)$ for $I_0\in X$.
\end{definition}

Assume that $f_i(i=1,2)$ satisfies 
$(\textbf{\textit{H1}})-(\textbf{\textit{H4}})$,
then  
$g_i\in\textit{H}(f_i)$ satisfies 
$(\textbf{\textit{H1}})-(\textbf{\textit{H4}})$.
Applying the standard semigroup theory for parabolic equations,
for any
given $g_i\in\textit{H}(f_i)$ and $(U_0,V_0)\in X^+$, there 
exists 
a 
unique solution $(U(\cdot,t;U_0,V_0,g_1,g_2),V(\cdot,t;U_0,V_0,g_1,g_2))$ for 
the 
following equation
\begin{equation}\label{equation-3}
	\left\{\begin{array}{ll}{U_{t}=D_{1} U_{x x}+g_1(x,t,U,V),} & {-L<x<L, 
			\enspace t>0,} \\ {V_{t}=D_{2} V_{x 
				x}+g_2(x,t,U,V),} & {-L<x<L, \enspace 
			t>0,} \\ {U(x, t)=V(x, t)=0,} & {x=-L \text { or } 
			x=L,\enspace 
			t>0,} 
	\end{array}\right.
\end{equation} 
for all 
$t>0$ with 
$U(\cdot,0;U_0,V_0,g_1,g_2)=U_0(x),V(\cdot,0;U_0,V_0,g_1,g_2)=V_0(x)$, where 
$X^+=\{(u,v)\in X\mid (u,v)\geq 0\}, X^{++}={\rm Int}(X^+)$.

Further, the system (\ref{equation-3}) generates a skew-product semiflow
\begin{equation}\label{equation-22}
	\begin{aligned}
		&{\Pi_t}: 
		X^{+}\times\textit{H}(f_1)\times\textit{H}(f_2)\longrightarrow 
		X^+\times\textit{H}(f_1)\times\textit{H}(f_2),\quad t\geq 0\\
		&\qquad\qquad\enspace(U_0,V_0,g_1,g_2)\mapsto(U(\cdot,t;U_0,V_0,g_1,g_2),V(\cdot,t;U_0,V_0,g_1,g_2),
		g_1\cdot t,g_2\cdot t),
	\end{aligned}
\end{equation}
where $g_i\cdot 
t(x,\cdot,U,V)=g_i(x,\cdot+t,U,V),i=1,2$.
It can be easily seen that ${\Pi_t}$ is continuous and compact by Lemma 
\ref{bounded}.

Next we introduce the 
definition of continuous separation for skew-product semiflow.
\begin{definition} [Definition 3.11,~\cite{shen2001}]
	The skew-product semiflow (\ref{equation-22}) is said to admit a 
	{continuous separation} if there are subspaces 
	$\left\{X_{1}(G)\right\}_{G 
		\in 
		H(F)}$ and $\left\{X_{2}(G)\right\}_{G \in H(F)}$ with the following 
	properties:
	\\1) $X=X_{1}(G) \oplus X_{2}(G)(G\in H(F))$ and $X_{1}(G), X_{2}(G)$ 
	vary 
	continuously for $G \in H(F)$;
	\\2) $X_{1}(G)=\operatorname{span}\{I(G)\},$ where $I(G) \in 
	\operatorname 
	X^{++}$ and $\|I(G)\|=1$ for $G \in H(F)$;
	\\3) $X_{2}(G) \cap X^{+}=\{0\}$ for every $G \in H(F)$;
	\\4) $ \Phi(G,t) X_{1}(G)=X_{1}(G \cdot t)$ 
	and 
	$\Phi(t, G) X_{2}(G) \subset$
	$X_{2}(G \cdot t)$ for any $t>0$ and $G \in H(F)$;
	\\5) There are $K_1>0$ and $\sigma>0$ such that for any $G \in H(F)$ and $w 
	\in 
	X_{2}(G)$ with $\|w\|=1$,
	$$
	\|\Phi(G,t) w\| \leq K_1 e^{-\sigma t}\|\Phi(G,t) I(G)\|, \quad t>0.
	$$
\end{definition}
\begin{lemma} \label{mono}
	Assume that $\lambda(A(x,t),L)$ is the principal  Lyapunov 
	exponent of (\ref{equation-2}), then 
	it is monotonically  increasing of $L\in 
	(0,\infty).$
\end{lemma}
\begin{proof}
	According to Lemma ${4.5}$ 
	(Part $\rm \uppercase\expandafter{\romannumeral3}$,
	~\cite{shen1998}), the skew-product semiflow 
	$\Pi_t$ generated by (\ref{equation-22})
	is strongly monotone in the sense that  
	$$(U(\cdot,t,U_0,V_0,g_1,g_2),U(\cdot,t,U_0,V_0,g_1,g_2))\in X^{++}$$ for 
	any 
	$t>0, (U_0,V_0)\in 
	X^{+},g_i\in \textit{H}(f_i)(i=1,2)$. Thus, by Theorem 4.4 of
	\cite{shen1998}, 
	the 
	skew-product semiflow (\ref{equation-22}) admits a continuous separation, 
	then 
	there exists $I_L:\textit{H}(A)\rightarrow X^{++} $ with $I_L=(U_L,V_L)$ 
	satisfying the 
	following properties:\\
	(\textit{a}) $I_L$ is continuous and $||I_L(\tilde A)||=1$ for any 
	$\tilde A\in\textit{H}(A)$;\\
	(\textit{b}) $\lambda(A,L)=\lim\limits_{t\rightarrow \infty}\frac{\ln 
		||I(\cdot,t,I_L,\tilde A)||}{t}=\lim\limits_{t\rightarrow 
		\infty}\frac{\ln 
		||\Phi(\tilde A,t)I_L(\tilde A)||}{t}$ for any $\tilde 
	A\in\textit{H}(A)$.
	
	Assume that $I(x,t,I_{L_{i}},A)$ for $i=1,2$ are the solutions for 
	(\ref{equation-3}) with $L=L_1, L_2$, respectively. Without loss of 
	generation, supposing 
	$0<L_1<L_2$, 
	then there is small 
	$\tau>0$ such that 
	$I_{L_2}\geq\tau I_{L_1}$ uniformly for $x\in [-L_1,L_1]$. According to 
	Comparison Principle, $I(x,t,I_{L_{2}},A)\geq I(x,t,\tau 
	I_{L_{1}},A)$ for $x\in [-L_1,L_1]$. In view of (\textit{a}) and 
	(\textit{b}), for any $\tilde A\in\textit{H}(A)$, it holds that 
	\begin{equation}\label{equation-23}
		\begin{aligned}
			\lambda(A,L_2)&=\lim\limits_{t\rightarrow\infty}\frac{\ln 
				||I(\cdot,t,I_{L_2},\tilde A)||}{t}
			\\&\geq\lim\limits_{t\rightarrow\infty}\frac{\ln 
				||I(\cdot,t,\tau I_{L_1},\tilde A)||}{t}
			\\&=\lim_{t\rightarrow\infty}\frac{\ln 
				||\Phi(\tilde A,t)\tau I_{L_1}(\tilde A)||}{t}
			\\&	=\lim_{t\rightarrow\infty}\frac{\ln (\tau
				||\Phi(\tilde A,t) I_{L_1}(\tilde A)||)}{t}
			\\&=\lim_{t\rightarrow\infty}\frac{\ln \tau+
				\ln||I(\cdot,t,I_{L_1},\tilde A)||}{t}
			\\&	=\lambda(A,L_1).
		\end{aligned}
	\end{equation}
	Thus, our proof is completed. 
\end{proof}

Considering that the infected domain $(g(t),h(t))$ is moving with respect to 
time  
$t$,  
we introduce the corresponding principal Lyapunov exponent 
$$\lambda(t):=\lambda(A,\frac{h(t)-g(t)}{2}), \enspace t\geq 0$$ for the 
following system
\begin{equation}\label{equation-t}
	\left\{\begin{array}{ll}{I_t=D(x,D)I+A(x,t)I,} & {g(t)<x<h(t),\enspace 
			t>0} 
		\\ 
		{I(h(t),t)=I(g(t),t)=0,}& {t>0,}\end{array}\right.
\end{equation}
where $-D(x,D)$ is a second-order strongly elliptic differential operator 
matrix of diagonal type with 
$D(x,D)=(D_i \partial_{ii})$ for $i=1,2$.	In view of the Lemmas \ref{bounded} 
and 
\ref{mono}, we can 
easily give the 
following result.
\begin{theorem}\label{ll}
	$\lambda(t)$ is monotonically increasing about $t$.
\end{theorem}

\section{The Long-Time Dynamics of WNv}\label{result2}
\noindent
In this section, we will discuss the long-time dynamical behaviors of the 
solution for (\ref{system-1}) and investigate the 
conditions determining 
the spreading permanently or vanishing eventually for this disease.

Firstly, we give the following definitions of vanishing and spreading for WNv.
\begin{definition}\quad
	The disease is { vanishing} if $h_{\infty}-g_{\infty}<\infty$ 
	and $$\lim\limits_{t\rightarrow 
		+\infty}||U(\cdot,t)||_{C(g(t),h(t))}=0,\lim\limits_{t\rightarrow 
		+\infty}||V(\cdot,t)||_{C(g(t),h(t))}=0;$$
	The disease is { spreading} if $h_{\infty}-g_{\infty}=\infty$ 
	and $$\lim\limits_{t\rightarrow +\infty}\inf\limits 
	||U(\cdot,t)||_{C(g(t),h(t))}>0, \lim\limits_{t\rightarrow 
		+\infty}\inf\limits||V(\cdot,t)||_{C(g(t),h(t))}>0.$$
\end{definition}

Next, for system (\ref{equation-3}), we recall a result from 
~\cite{shen2004lyapun} (Theorem 
A) which 
will be applied in proving Theorem 
\ref{dichotomy} and Theorem \ref{threshold}.
\begin{lemma}\label{propA}
	Let matrix function $A(x,t)$ be defined by 
	(\ref{A(x,t)}). For any given $g_i\in \textit{H}(f_i)$ for $i=1,2$. Let 
	$\left(U(\cdot,t;U_0,V_0,g_1,g_2), V(\cdot,t;U_0,V_0,g_1,g_2)\right)$ be 
	the solution of ($\ref{equation-3}$), then the followings
	hold.\\
	\textbf{(1)} If	$\lambda(A,L)<0$, then $$\lim\limits_{t\rightarrow 
		\infty}||U(\cdot,t;U_0,V_0,g_1,g_2)||=0, \lim\limits_{t\rightarrow 
		\infty}||V(\cdot,t;U_0,V_0,g_1,g_2)||=0$$ uniformly for 
	$g_i\in\textit{H}(f_i)$. Further, $\lim\limits_{t\rightarrow 
		\infty}||U(\cdot,s+t;U_0,V_0,s)||=0$ and $ \lim\limits_{t\rightarrow 
		\infty}||U(\cdot,s+t;U_0,V_0,s)||=0$ uniformly for $s\in \mathbb R$.\\
	\textbf{(2)} If	$\lambda(A,L)>0$, there exist $U_L:\textit{H}(f_1)\times 
	\textit{H}(f_2)\longrightarrow C([-L,L])$ and $V_L:\textit{H}(f_1)\times 
	\textit{H}(f_2)\longrightarrow C([-L,L])$ such that $U_L(g_1,g_2)$ and 
	$V_L(g_1,g_2)$ 
	are continuous for $g_i\in\textit{H}(f_i)$ and 
	$U(\cdot,t;U_L,V_L,g_1,g_2)=U_L(g_1\cdot 
	t,g_2\cdot t)(\cdot)$, $V(\cdot,t;U_L,V_L,g_1,g_2)=V_L(g_1\cdot 
	t,g_2\cdot 
	t)(\cdot)$. Meanwhile, 
	it holds that 
	\[\lim_{t\rightarrow\infty}||U(\cdot,t;U_0,V_0,g_1,g_2)-U(\cdot,t;U_L(g_1,g_2),V_L(g_1,
	g_2),g_1,g_2)||=0,\]
	\[\lim_{t\rightarrow\infty}||V(\cdot,t;U_0,V_0,g_1,g_2)-V(\cdot,t;U_L(g_1,g_2),V_L(g_1,
	g_2),g_1,g_2)||=0\]
	uniformly in $g_i\in\textit{H}(f_i)$ for any $(U_0,V_0)\in X^{+}\backslash 
	\{0\}$.
	Further, $U^*_L(x,t):=U_L(f_1\cdot t,f_2\cdot f_2)(x)$ and 
	$V^*_L(x,t):=V_L(f_1\cdot t,f_2\cdot f_2)(x)$ are uniformly almost periodic 
	in $t\in \mathbb R$. Moreover,  for any $(U_0,V_0)\in X^+\backslash 
	\{0\}$, 
	it holds 
	that
	\[\lim_{t\rightarrow\infty}||U(\cdot,s+t;U_0,V_0,s)-U^*_L(\cdot,s+t)||=0,
	\lim_{t\rightarrow\infty}||V(\cdot,s+t;U_0,V_0,s)-V^*_L(\cdot,s+t)||=0\]
	uniformly for  $s\in \mathbb R$, where 
	$U(\cdot,s+t;U_0,V_0,s)=U(\cdot,t;U_0,V_0,f_1\cdot s,f_2\cdot s), 
	V(\cdot,s+t;U_0,V_0,s)=V(\cdot,t;U_0,V_0,f_1\cdot s,f_2\cdot s)$.	
\end{lemma}

\begin{lemma}\label{lemma-inf}
	Assume that $(\textbf{H1})-(\textbf{H5})$ hold. Take $L\geq L^*$, then  
	\[\inf_{x\in[-L,L],\tilde x\in \mathbb R,\atop g_i\in 
		{H(f_i)}}U^*(x,0;\tilde x,g_1,g_2,L)>0,
	\inf_{x\in[-L,L],\tilde x\in \mathbb R,\atop g_i\in 
		{H(f_i)}}V^*(x,0;\tilde x,g_1,g_2,L)>0,\] for $i=1,2.$
	Where $(U^*(x,t;\tilde x,g_1,g_2,L),V^*(x,t;\tilde x,g_1,g_2,L))$ is the 
	unique positive 
	almost 
	periodic solution of the following system,
	\begin{equation}
		\left\{\begin{array}{ll}{U_{t}=D_{1} U_{x x}+g_1(x+\tilde x,t,U,V),} & 
			{-L<x<L, \tilde x\in\mathbb R,
				t>0,} \\ {V_{t}=D_{2} V_{x 
					x}+g_2(x+\tilde x,t,U,V),} & {-L<x<L, \tilde x\in\mathbb 
				R, 
				t>0,} \\ {U(x, t)=V(x, t)=0,} & {x=-L \enspace {\rm or } 
				\enspace
				x=L,
				t>0.} 
		\end{array}\right.
	\end{equation} 
\end{lemma}
Indeed, we can see that \[(U^*(x,t;\tilde x,g_1,g_2,L),V^*(x,t;\tilde 
x,g_1,g_2,L))=
(U^*(x,0;\tilde x,g_1\cdot t,g_2\cdot t,L),V^*(x,0;\tilde x,g_1\cdot t 
,g_2\cdot t,L)).\]
\begin{proof}
	The proof of this lemma can refer to Lemma 4.1 in \cite{shen2015free}, it 
	can be proved by making a minor modification, so we omit the detailed 
	proof. 
\end{proof}

Considering the dependence of boundary functions $g(t)$ and $h(t)$ on $\mu$, 
denote 
$$h_{\mu}(t):=h(t)=h(t;U_0,V_0,h_0) \enspace {\rm and}\enspace 
g_{\mu}(t):=g(t)=g(t;U_0,V_0,h_0)$$  
with $h(0)=h_0,g(0)=-h_0$. Then the following result holds.
\begin{lemma}\label{hmu(t)}
	For all $t>0$, $h_{\mu}(t)$ is strictly monotonically increasing 
	in $\mu$, and $g_{\mu}(t)$ is strictly monotonically decreasing in $\mu$.
\end{lemma}
\begin{proof}
	We will prove this lemma mainly by Comparison Principle.	Assume that 
	$(U_1, V_1; 
	g_{{\mu}_1}, h_{{\mu}_1})$ and $(U_2, V_2; 
	g_{{\mu}_2}, h_{{\mu}_2})$ are the two 
	solutions 
	for problem $(\ref{system-1})$.
	For simplification, we only need to compare  $h_{\mu_1}(t)$ with 
	$h_{\mu_2}(t)$, 
	then we can similarly obtain the strict monotonicity of $g_\mu(t)$.

	Without loss of generality, assume that $0<\mu_1<\mu_2$, then 
	\begin{equation}
		h^{\prime}_{\mu_1}(t)=-\mu_1 {U_1}_x(h_{\mu_1}(t),t)<-\mu_2 
		{U_1}_x(h_{\mu_1}(t),t).
	\end{equation}
	By Lemma \ref{lemma-comparison}, it follows 
	$h_{\mu_1}(t)\leq h_{\mu_2}(t)$ for all $t\in[0,\infty)$.
	
	Now it is our turn to prove that $h_{\mu_1}(t)<h_{\mu_2}(t)$ in 
	$[0,\infty)$.
	On the contrary, assume that positive time $T^*$ is the first time such 
	that 
	$h_{{\mu}_1}(t)
	<h_{{\mu}_2}(t)$ for $t\in(0,T^*)$ and $h_{{\mu}_1}(T^*)
	=h_{{\mu}_2}(T^*)$, then 
	\begin{equation}\label{hh}
		h^{\prime}_{{\mu}_1}(T^*)
		\geq h^{\prime}_{{\mu}_2}(T^*).
	\end{equation}
	Let 
	$\Sigma_{T^{*}}:=\left\{(x,t) \in \mathbb{R}^2\mid 0 \leq 
	x<h_{\mu_{1}}(t), 0<t \leq T^{*}\right\}.
	$
	Applying the strong maximum principle to $U_1$ and $U_2$, it follows that
	$U_{1}( x,t)<U_{2}(x,t)$ 
	in 
	$ 
	\Sigma_{T^{*}}$. Let $H(x,t)=U_{2}(x,t)-U_{1}(x,t),$ then 
	$H(x,t)>0 
	\enspace{\rm 
		for}\enspace (x,t)\in 
	\Sigma_{T^{*}}$ and $H\left(h_{\mu_{1}}\left(T^{*}\right),T^{*}\right)=0 .$ 
	Then, we can get that 
	$H_{x}\left(h_{\mu_{1}}\left(T^{*}\right),T^{*}\right)<$
	$0$.  In view of $\left(U_{i}\right)_{x}\left( 
	h_{\mu_{1}}\left(T^{*}\right),T^{*}\right)<0$ and $\mu_{1}<\mu_{2},$ then
	$$
	-\mu_{1}\left(U_{1}\right)_{x}\left(
	h_{\mu_{1}}\left(T^{*} 
	\right),T^{*}\right)<-\mu_{2}\left(U_{2}\right)_{x}\left( 
	h_{\mu_{2}}\left(T^{*}\right),T^{*}\right).
	$$
	Therefore,
	$h_{\mu_{1}}^{\prime}\left(T^{*}\right)<h_{\mu_{2}}^{\prime}\left(T^{*}\right)
	,$ which yields a contradiction to (\ref{hh}). Thus,  
	$h_{\mu}(t)$ is strictly monotonically increasing about $\mu$ for all $t>0$.
	
	Similarly, we can get that $-g_{\mu_{1}}(t)<-g_{\mu_{2}}(t)$ for all $t>0$. 
	Therefore, our proof is completed.
\end{proof}

In the rest of this section, for any given $(U_0,V_0)$ satisfying 
(\ref{system-2}), let  $$(U(x,t;U_0,V_0,h_0),
V(x,t;U_0,V_0,h_0))$$ denote the solution of system $(\ref{system-1})$ with 
$$U(x,0;U_0,V_0,h_0)=U_0, V(x,0;U_0,V_0,h_0)=V_0,h(0)=h_0,g(0)=-h_0.$$ 
\begin{theorem}\label{hg}
	If $h_\infty-g_\infty<\infty$, then $\lim\limits_{t\rightarrow 
		\infty}h^{\prime}(t,U_0,V_0,h_0)=0, \lim\limits_{t\rightarrow 
		\infty}g^{\prime}(t,U_0,V_0,h_0)=0$.
\end{theorem}
\begin{proof}
	Now we are going to prove the case of 
	$h^{\prime}(t,U_0,V_0,h_0)$.
	On the contrary,
	assume that there exists a positive sequence $\{t_n\}$ with
	$\lim\limits_{n\rightarrow\infty}t_n= \infty 
	$ 
	such 
	that 
	\begin{equation}\label{h'}
		\lim\limits_{n\rightarrow \infty} h^{\prime}(t_n,U_0,V_0,h_0)>0.
	\end{equation}
	Let \[h_n(t)=h(t+t_n,U_0,V_0,h_0),\enspace {\rm for} \enspace t\geq 0,\] 
	then 
	$\lim\limits 
	_{n\rightarrow \infty}h_n(t)=h_\infty$ uniformly for $t\geq 0$. According 
	to 
	Lemma 
	\ref{bounded}, we can get that $\{h^{\prime}_n(t)\}$ is uniformly 
	bounded and equicontinuous on $[0,\infty)$. By {Arzela-Ascoli theorem}, 
	there exists $h^{*}(t)$ such that 
	$\lim\limits_{n\rightarrow\infty}h^{\prime}_n(t)=h^*(t)$ uniformly in any 
	bounded 
	sets of $[0,\infty)$. Since $\lim\limits 
	_{n\rightarrow \infty}h_n(t)=h_\infty<\infty,$ then $h^*(t)\equiv 0, $ 
	which 
	implies that 
	$\lim\limits_{n\rightarrow\infty}h^{\prime}(t_n,U_0,V_0,h_0)=0$. 
	It is a 
	contradiction to (\ref{h'}).
	Similarly, $\lim\limits_{t\rightarrow 
		\infty}g^{\prime}(t,U_0,V_0,h_0)=0$.
\end{proof}
\begin{theorem}\label{tm5}
	Assume that  
	$(\textit{\textbf{H1}})-(\textit{\textbf{H5}})$ hold. If 
	$h_\infty-g_\infty<\infty$, then 
	$$\lim\limits_{t\rightarrow 
		+\infty}U(x,t;U_0,V_0,h_0)=0 \enspace {\rm and} \enspace 
	\lim\limits_{t\rightarrow 
		+\infty}V(x,t;U_0,V_0,h_0)=0$$ uniformly in $x\in [g_\infty,h_\infty]$. 
	That is, the disease will vanish.
\end{theorem}
\begin{proof}
	Let $f_i$ be defined as in (\ref{fi}) for $i=1,2$.
	then $f_1$ and $f_2$ satisfy 
	$(\textit{\textbf{H1}})-(\textit{\textbf{H4}})$ and $A(x,t)$ defined by 
	(\ref{A(x,t)}) satisfies $(\textit{\textbf{H5}})$. 
	If 
	$h_\infty-g_\infty<\infty$, it is easily to show that $h_\infty<\infty$ 
	and 
	$g_\infty>-\infty$. 
	
	According to regularity and the priori estimates about parabolic 
	equations (\cite{henry1981}), considering 
	the system $(\ref{system-1})$,  for any given sequence $\{t_n\}$ satisfying 
	$ 
	t_{n}\rightarrow \infty$ as $n\rightarrow \infty$, there exists a 
	subsequence 
	$\{t_{n_k}\}$ satisfying $ t_{n_k}\rightarrow \infty$ as $k\rightarrow 
	\infty$,
	$(\hat U^*(x,t), \hat V^*(x,t))\in (C([g_\infty,h_\infty]\times \mathbb 
	R))^2$ and 
	$g^*_i\in 
	\textit{H}(f_i)$ for $i=1,2$ such that $f_i\cdot t_{n_k}\rightarrow g^*_i,$
	\begin{equation}\label{equation-16}
		\lim\limits_{{k}\rightarrow 
			\infty}||U(\cdot 
		,t+t_{n_k};U_0,V_0,h_0)-\hat
		U^*(\cdot,t)||_{C^1([g(t+t_{n_k},h(t+t_{n_k})])}
		=0
	\end{equation}
	and
	\begin{equation}\label{equation-17}
		\lim\limits_{{k}\rightarrow\infty}||V(\cdot 
		,t+t_{n_k};U_0,V_0,h_0)-\hat V^*(\cdot,t)||
		_{C^1([g(t+t_{n_k},h(t+t_{n_k})])}=0,
	\end{equation}
	where $(\hat U^*(x,t), \hat V^*(x,t))$ is the entire solution for the 
	following 
	system,
	\begin{equation}\label{system-25}
		\left\{\begin{array}{ll}{U_{t}=D_{1} U_{x x}+g^*_1(x,t,U,V),} & 
			{g_\infty<x<h_\infty, 
			} \\ {V_{t}=D_{2} V_{x 
					x}+g^*_2(x,t,U,V),} & {g_\infty<x<h_\infty, 
			} \\ {U(x, t)=V(x, t)=0,} & {x=g_\infty \text { or } 
				x=h_\infty.} 
		\end{array}\right.
	\end{equation}
	Next we 
	accomplish the proof of this theorem by
	two steps. 
	
	\textbf{\textbf{Step 1}} To show $h_\infty-g_\infty\leq 2L^*$ 
	following 
	from 
	$h_\infty-g_\infty<\infty$.
	
	On the contrary, assume that $h_\infty-g_\infty\in (2L^*,\infty) $, then 
	there 
	exist 
	$t^*>0$ and $\epsilon>0$ such that 
	$h(t)-g(t)>h_\infty-g_\infty-2\epsilon>2L^*$ 
	for 
	$t\geq t^*$, thus, by (\textit{\textbf{H5}}) and Theorem \ref{ll},
	$\lambda(t)>0$.
	For the following system
	\begin{equation}\label{system-9}
		\left\{\begin{array}{ll}{U_{t}=D_{1} U_{x x}+f_1(x,t,U,V),} & 
			{g_\infty+\epsilon<x<h_\infty-\epsilon, 
				\enspace t>0,} \\ {V_{t}=D_{2} V_{x 
					x}+f_2(x,t,U,V),} & 
			{g_\infty+\epsilon<x<h_\infty-\epsilon, 
				\enspace 
				t>0,} \\ {U(x, t)=V(x, t)=0,} & {x=g_\infty+\epsilon 
				\text { or } 
				x=h_\infty-\epsilon,\enspace 
				t>0,} 
		\end{array}\right.
	\end{equation}
	by Comparison Principle, we can get that 
	\[U(\cdot,t+t^*;U_0,U_0,h_0)\geq \tilde 
	U(\cdot,t+t^*;U(\cdot,t^*;U_0,U_0,h_0),V(\cdot,t^*;U_0,U_0,h_0),t^*)\]and  
	\[U(\cdot,t+t^*;U_0,U_0,h_0)\geq \tilde 
	V(\cdot,t+t^*;U(\cdot,t^*;U_0,U_0,h_0),V(\cdot,t^*;U_0,U_0,h_0),t^*).\]
	Where $$(\tilde 
	U(\cdot,t+t^*;U(\cdot,t^*;U_0,U_0,h_0),V(\cdot,t^*;U_0,U_0,h_0),t^*), 
	\tilde 
	V(\cdot,t+t^*;U(\cdot,t^*;U_0,U_0,h_0),V(\cdot,t^*;U_0,U_0,h_0),t^*))$$ is 
	the 
	solution of $(\ref{system-9})$ with $$\tilde 
	U(\cdot,t^*;U(\cdot,t^*;U_0,U_0,h_0),V(\cdot,t^*;U_0,U_0,h_0),t^*)=
	U(\cdot,t^*;U_0,U_0,h_0),$$ $$\tilde 
	V(\cdot,t^*;U(\cdot,t^*;U_0,U_0,h_0),V(\cdot,t^*;U_0,U_0,h_0),t^*))=
	V(\cdot,t^*;U_0,U_0,h_0).$$ In  view of  
	Lemma \ref{propA}, the system $(\ref{system-9})$ admits a positive 
	almost time periodic solution $(U_{\epsilon}(x,t), 
	V_{\epsilon}(x,t))$. Moreover, for any $(U_0,V_0)\in X^{++}$, it holds that
	\begin{equation}\label{equality-2}
		\lim_{t\rightarrow \infty}||\tilde 
		U(\cdot,t+t^*;U_0,V_0,t^*)-U_{\epsilon}(\cdot,t+t^*)||=0
	\end{equation}
	and
	\begin{equation}\label{equality-3}
		\lim_{t\rightarrow \infty}||\tilde
		V(\cdot,t+t^*;U_0,V_0,t^*)-V_{\epsilon}(\cdot,t+t^*)||=0.
	\end{equation}
	By Comparison Principle, combining $(\ref{equality-2})$ and 
	$(\ref{equality-3})$, we get 
	\[\hat U^*(x,t)>0, \hat V^*(x,t)>0,  x\in (g_\infty,h_\infty),t\in \mathbb 
	R,\]
	which implies $\hat U^*_x(h_\infty,t)<0, \hat V^*_x(h_\infty,t)<0$.
	Therefore, 
	\[\lim_{t\rightarrow \infty}\sup U_x(h(t),t;U_0,V_0,h_0)<0,
	\]
	it	implies
	\[
	\lim_{t\rightarrow \infty}\inf h^{\prime}(t)=\lim_{t\rightarrow \infty}
	\inf-\mu 
	U_x(h(t),t;U_0,V_0,h_0)>0,
	\]
	which is contradict to Theorem \ref{hg}. Thus, we can obtain that 
	$h_\infty-g_\infty<\infty$ gives 
	$h_\infty-g_\infty\leq 2L^*$.

	\textbf{\textbf{Step 2}} To show that if $h_\infty-g_\infty<\infty$, then 
	\begin{equation}
		\lim\limits _{t 
			\rightarrow \infty}\left\|U\left(\cdot,t; 
		U_{0}, 
		V_0, 
		h_{0}\right)\right\|_{C([g(t), h(t)])}=0,
		\lim\limits _{t \rightarrow \infty}\left\|V\left(\cdot,t; U_{0}, V_0,
		h_{0}\right)\right\|_{C([g(t), h(t)])}=0.
	\end{equation}
	Let
	\[
	\tilde{u}_{0}(x)=\left\{\begin{array}{ll}
		U_{0}(x), & \text { for } -h_0 \leq x \leq h_{0}, \\
		0, & \text { for } |x|>h_{0}.
	\end{array}\right.
	\]
	\[
	\tilde{v}_{0}(x)=\left\{\begin{array}{ll}
		V_{0}(x), & \text { for } -h_0 \leq x \leq h_{0}, \\
		0, & \text { for } |x|>h_{0}.
	\end{array}\right.
	\]
	
	Assume that 
	$(\bar{u}(x,t),\bar{v}(x,t))$ is the
	solution of the problem
	\[
	\left\{\begin{array}{ll}
		\bar{u}_{t}=D_1\bar{u}_{x x}+f_1( x,t, \bar{u},\bar v), & 
		g_\infty<x<h_{\infty},t>0, \\
		\bar{v}_{t}=D_2\bar{v}_{x x}+f_2( x,t, \bar{u},\bar v), & 
		g_\infty<x<h_{\infty},t>0,\\
		\bar{u}(g_\infty,t)=\bar{u}\left( h_{\infty},t\right)=0, & t>0, \\
		\bar{v}( g_\infty,t)=\bar{v}\left( h_{\infty},t\right)=0, & t>0,\\
		\bar{u}( x,0)=\tilde{u}_{0}(x), \bar{v}( x,0)=\tilde{v}_{0}(x), & 
		g_\infty 
		\leq x \leq h_{\infty}.
	\end{array}\right.
	\]
	
	Applying the Lemma \ref{lemma-comparison}, we can get that
	$$
	\bar{u}(x,t) \geq U\left(x,t ; U_{0}, V_0, h_{0}\right) \geq 0, 
	\bar{v}(x,t) 
	\geq 
	V\left( x,t 
	; U_{0}, V_0, h_{0}\right) \geq 0,\enspace{\rm for}\enspace 
	x \in[g(t), h(t)],\enspace t>0.
	$$
	If $h_{\infty}-g_\infty <2L^{*},$ assuming (\textit{\textbf{H5}}), then 
	$\lambda\left(A, 
	\frac{h_{\infty}-g_\infty}{2}\right)<0$.  By
	Lemma \ref{propA}, 
	$\lim\limits_{t\rightarrow\infty}(\bar{u},\bar{v})=(0,0)$
	uniformly for $x \in\left[g
	_\infty, 
	h_{\infty}\right]$. Hence, $$\lim _{t \rightarrow 
		\infty}\left\|U\left(\cdot,t; U_{0}, V_0,h_{0}\right)\right\|_{C([g(t), 
		h(t)])}=0, \lim _{t \rightarrow 
		\infty}\left\|V\left(\cdot,t; U_{0}, V_0, 
	h_{0}\right)\right\|_{C([g(t), 
		h(t)])}=0.$$
	If $h_{\infty}-g_\infty=2L^{*},$ without loss of generality, assume that 
	$\lim 
	\limits_{t \rightarrow 
		\infty}\left\|U\left(\cdot,t ; U_{0}, V_0, 
	h_{0}\right)\right\|_{C([g(t), 
		h(t)])} 
	\neq 
	0$, then there exist a sequence
	$\{\breve s_{n}\}$ with $\breve s_n \longrightarrow\infty$ as $n\rightarrow 
	\infty$, 
	$(\breve U^{*}(x), 
	\breve V^{*}(x))$ with $U^*\geq,\not\equiv 0$ and $\breve g_i^{*} \in 
	\textit{H}(f_i)$ such 
	that $\lim\limits_{n\rightarrow \infty}f_i \cdot \breve s_{n}=\breve 
	g_i^{*}$ and
	$$\lim\limits_{n\rightarrow \infty}\left\|U\left(\cdot,\breve s_{n} ; 
	U_{0}, V_0,
	h_{0}\right)-\breve U^{*}(\cdot)\right\|_{C\left(\left[g(\breve s_n), 
		h\left(\breve s_{n}\right)\right]\right)} = 0,$$
	$$\lim\limits_{n\rightarrow \infty}\left\|V\left(\cdot,\breve s_{n} ; 
	U_{0}, V_0,
	h_{0}\right)-\breve V^{*}(\cdot)\right\|_{C\left(\left[g(\breve s_n), 
		h\left(\breve s_{n}\right)\right]\right)} = 0.$$
	It follows that  
	$(U\left(\cdot,t ; \breve U^{*}, \breve V^*, 
	\breve g_1^{*},\breve g^*_2\right),V\left(\cdot,t ; 
	\breve U^{*}, \breve V^*,\breve g_1^{*}, \breve g^*_2\right))$ is the 
	entire solution for 
	the following equation,
	\begin{equation}
		\left\{\begin{array}{ll}
			u_{t}=D_1u_{x x}+ \breve g_1^{*}( x,t, u,v), \quad 
			g_\infty<x<h_\infty, \\
			v_{t}=D_2v_{x x}+ \breve g_2^{*}( x,t, u,v), \quad 
			g_\infty<x<h_\infty, \\
			u(g_\infty,t)=u( h_\infty,t)=0,\\
			v( g_\infty,t)=v( h_\infty,t)=0.
		\end{array}\right.
	\end{equation}
	Applying Hopf lemma to $U(h_\infty,t;\breve U^*, \breve V^*,\breve g^*_1, 
	\breve g^*_2)$ 
	and 
	$U(g_\infty,t;\breve U^*,\breve V^*, \breve g^*_1, \breve g^*_2)$, we can 
	get 
	that $$U_{x}\left( 
	h_\infty,t ; \breve U^{*}, \breve V^*,
	\breve g^{*}_1, \breve g^*_2\right)<0, U_{x}\left( 
	g_\infty,t ; \breve U^{*}, \breve V^*,
	\breve g^{*}_1, \breve g^*_2\right)>0,$$
	which implies 
	\[
	\lim _{n \rightarrow \infty} h^{\prime}\left(\breve s_{n}\right)=-\lim _{n 
		\rightarrow 
		\infty} \mu U_{x}\left( h\left(\breve s_{n}\right),\breve s_{n} ; 
	U_{0}, V_0, 
	h_{0}\right)>0
	\]
	and
	\[
	\lim\limits _{n \rightarrow \infty} 
	g^{\prime}\left(\breve s_{n}\right)=-\lim\limits 
	_{n 
		\rightarrow 
		\infty} \mu U_{x}\left( g\left(\breve s_{n}\right),\breve s_{n} ; 
	U_{0}, V_0, 
	h_{0}\right)<0.
	\]
	This is contradict to Theorem \ref{hg}. Thus, our proof is 
	completed.
\end{proof}

\begin{rem}
	From the proof of the above theorem, we can obtain that the densities of 
	infected populations will decay to 0 and the eventually 
	infected domain is no more than $2L^*$ when the disease vanishes.
\end{rem}

The following theorem gives the long-time asymptotic behavior as the spreading 
happens, which is the sharp distinction for our spatial heterogeneous and 
time almost periodic WNv model. 
\begin{theorem}\label{tm6}
	Assume that (\textit{\textbf{H}}1)-(\textit{\textbf{H}}5) hold. For any 
	given 
	$h_0$ and $(U_0,V_0)$ satisfying (\ref{system-2}), let 
	$$(U(x,t;U_0,V_0,h_0),V(x,t;U_0,V_0,h_0))$$ be the solution for 
	(\ref{system-1}). If 
	$h_\infty-g_\infty=\infty,$ then 
	\begin{equation}\label{uv}
		\lim\limits_{t\rightarrow 
			+\infty}U(x,t;U_0,V_0,h_0)-U^*(x,t)=0,
		\lim\limits_{t\rightarrow +\infty}V(x,t;U_0,V_0,h_0)-V^*(x,t)=0
	\end{equation}
	locally
	uniformly 
	for 
	$x\in\mathbb{R},$ where $(U^*(x,t),V^*(x,t))$ is the unique positive almost 
	periodic 
	solution of the  system (\ref{system-10}). That is, the disease will spread.
\end{theorem}
\begin{proof}
	Firstly, we aim to show that $h_\infty=\infty$ and 
	$g_\infty=-\infty$ 
	when 
	$h_\infty-g_\infty=\infty$. On 
	the 
	contrary, assume that $g_\infty=-\infty$ and 
	$h_\infty<\infty$. According to Theorem \ref{hg}, it is easily 
	to 
	yield a contradiction to $\lim\limits_{t\rightarrow \infty}h^{\prime}(t)=0$.
	Therefore, $g_\infty=-\infty$ and $h_\infty=\infty$.
	
	Next we will prove (\ref{uv}). 
	Let $U_0:=N_1,V_0:=N_2$, then by Comparison Principle, 
	$U(x,t;N_1,N_2,f_1\cdot (-t),f_2\cdot(-t))$ and $V(x,t;N_1,N_2,f_1\cdot 
	(-t),f_2\cdot(-t))$ decrease in $t\in\mathbb R$.
	Take	
	$$U^*(f_1,f_2)(x):=\lim\limits_{t\rightarrow\infty}U(x,t;N_1,N_2,f_1\cdot(-t),
	f_2\cdot(-t)), 
	V^*(f_1,f_2)(x):=\lim\limits_{t\rightarrow\infty}V(x,t;N_1,N_2,f_1\cdot(-t),
	f_2\cdot(-t))$$ for $x\in \mathbb R$. And it follows
	\[U(\cdot,t;U^*(f_1,f_2)(x),V^*(f_1,f_2)(x),f_1,f_2)=U^*(f_1\cdot 
	t,f_2\cdot 
	t)(\cdot),\]
	\[V(\cdot,t;U^*(f_1,f_2)(x),V^*(f_1,f_2)(x),f_1,f_2)=V^*(f_1\cdot 
	t,f_2\cdot 
	t)(\cdot),\]
	where $(U(x,t;N_1,N_2,f_1,
	f_2), V(x,t;N_1,N_2,f_1,
	f_2))$ is the solution for (\ref{equation-3}) for $U_0=N_1,V_0=N_2$ and 
	$L=\infty$. Let
	$U_{L}(f_1,f_2)(x)$ and $V_{L}(f_1,f_2)(x)$ be in Lemma 
	\ref{propA},  
	then for any fixed $x, 
	U_{L}(f_1,f_2)(x)$ and $V_{L}(f_1,f_2)(x)$ are  
	increasing in $L$. Applying the Comparison Principle and Lemma 
	\ref{lemma-inf},  we can obtain that
	\begin{equation}\label{equation-21}
		\lim _{L \rightarrow \infty} U_{L}(f_1,f_2)(x)=
		U^{*}(f_1,f_2)(x),\lim _{L \rightarrow \infty} V_{L}(f_1,f_2)(x)=
		V^{*}(f_1,f_2)(x)
	\end{equation}
	locally uniformly for $x\in \mathbb R$. 
	
	For any $T>0$ satisfying $h(T)-g(T)>2L^{*}$, denote $U\left( 
	\cdot,T 
	; 
	U_{0}, V_0, h_{0}\right):=U(\cdot,T)$ and $V\left( 
	\cdot,T 
	; 
	U_{0}, V_0, h_{0}\right):=V(\cdot,T)$, we can get   
	$$
	U\left(x,t+T ; U_{0},V_0, h_{0}\right) \geq U_{L}\left(x,t ; U\left( 
	\cdot,T 
	\right), V\left( \cdot, T 
	\right),f_1 \cdot T,f_2 \cdot T\right) \quad {\rm for} 
	\quad  
	t 
	\geq 0,
	$$
	$$
	V\left(x,t+T ; U_{0},V_0, h_{0}\right) \geq V_{L}\left(x,t; U(\cdot,T), 
	V\left( 
	\cdot, T
	\right), f_1 \cdot T,f_2 \cdot T\right) \quad {\rm for} 
	\quad  
	t 
	\geq 0,
	$$
	where $(U_{L}\left( x,t ; U\left(\cdot,T \right), V\left( \cdot,T \right), 
	f_1 
	\cdot 
	T,f_2 \cdot T\right), V_{L}\left( x,t ; U\left( \cdot,T \right),
	V\left( \cdot ,T
	\right), f_1 \cdot 
	T,f_2 \cdot T\right)$  is the solution of following system
	\begin{equation}
		\left\{\begin{array}{ll}{U_{t}=D_{1} U_{x x}+f_1\cdot T(x,t,U,V),} & 
			{g\left(T ; U_{0},V_0,  
				h_{0}\right)<x<h\left(T ; U_{0},V_0,  
				h_{0}\right), 
				\enspace t>0,} \\ {V_{t}=D_{2} V_{x 
					x}+f_2\cdot T(x,t,U,V),} & {g\left(T ; U_{0},V_0,  
				h_{0}\right)<x<h\left(T ; U_{0},V_0,  
				h_{0}\right), \enspace 
				t>0,} \\ {U(x, t)=V(x, t)=0,} & {x=g\left(T ; U_{0},V_0,  
				h_{0}\right) \text { or } 
				x=h\left(T ; U_{0},V_0,  
				h_{0}\right),\enspace 
				t>0} 
		\end{array}\right.
	\end{equation}
	with $L=$ $\dfrac{h\left(T ; U_{0},V_0, 
		h_{0}\right)-g\left(T ; U_{0},V_0,  
		h_{0}\right)}{2},$ $$U_L\left( x,0 ; U\left( \cdot,T \right), 
	V\left( \cdot,T \right),
	f_1
	\cdot T,f_2
	\cdot T\right)=U\left( x,T ; U_{0},V_0, h_{0}\right)$$  and 
	$$V_L\left( x,0 ; 
	U\left( 
	\cdot,T\right), V\left( \cdot,T\right),
	f_1 
	\cdot T, f_2 
	\cdot T\right)=V\left( x,T ; U_{0},V_0, h_{0}\right).$$ According to  
	Lemma 
	\ref{propA},
	$$
	\lim\limits_{t\rightarrow \infty}U_L\left( x,t ; U\left( \cdot,T\right), 
	V\left( \cdot,T\right), f_1 \cdot 
	T,f_2 \cdot 
	T\right)-U_L(f_1 \cdot(t+T),f_2 \cdot(t+T))(x)=0,
	$$
	$$
	\lim\limits_{t \rightarrow \infty}V_L\left( x,t ; U\left( \cdot,T\right), 
	V\left( \cdot,T\right), f_1 \cdot 
	T,f_2 \cdot 
	T\right)-V_L(f_1 \cdot(t+T),f_2 \cdot(t+T))(x)=0 
	$$
	uniformly for $x$ in $[g\left(T ; U_{0},V_0,  
	h_{0}\right),h\left(T ; U_{0},V_0,  
	h_{0}\right)]$. In view of (\ref{equation-21})
	$$
	\lim\limits_{L \rightarrow \infty}U_L(f_1 \cdot(t+T),f_2 
	\cdot(t+T))(x)=U^{*}(f_1 \cdot(t+T),f_2 \cdot(t+T))(x) 
	,
	$$
	$$
	\lim\limits_{L \rightarrow \infty}V_L(f_1 \cdot(t+T),f_2 
	\cdot(t+T))(x)=V^{*}(f_1 \cdot(t+T),f_2 \cdot(t+T))(x) 
	$$
	uniformly for $x$ in any bounded sets of  $\mathbb 
	R.$ By Comparison Principle,  $$(U(x,t;U_0,V_0,h_0),V(x,t;U_0,V_0,h_0))\geq 
	(U_L(x,t;U_0,V_0,h_0),V_L(x,t;U_0,V_0,h_0))$$ uniformly for $(x,t)\in 
	[g(t),h(t)]\times 
	[0,\infty)$ . Then we can get
	$$
	\lim\limits_{t \rightarrow \infty}U\left( x,t ; U_{0}, V_0, 
	h_{0}\right)-U^{*}(f_1 \cdot t,f_2 \cdot t)(x)=0 
	$$
	and
	$$
	\lim\limits_{t \rightarrow \infty}V\left( x,t ; U_{0}, V_0, 
	h_{0}\right)-V^{*}(f_1 \cdot t,f_2 \cdot t)(x)=0 
	$$
	locally uniform for $x\in\mathbb R.$ Therefore, $(U^*(f_1 \cdot t,f_2 \cdot 
	t)(x),V^*(f_1 \cdot t,f_2 \cdot t)(x))$ is the solution for 
	(\ref{system-10}). Applying the similar method in Proposition 4.1(3) of 
	\cite{shen2015free}, the uniqueness of the solution 
	can be easily proved.
	Further, by Lemma \ref{lemma-inf}, we can obtain that
	\begin{equation}\label{ii}
		\inf _{x\in \mathbb R,t\in\mathbb 
			R^+}U^*(f_1\cdot t,f_2\cdot t)(x)>0,\quad 
		\inf_{x\in \mathbb R, t\in\mathbb 
			R^+}V^*(f_1\cdot t,f_2\cdot t)(x)>0.
	\end{equation}
	Take $U^*(x,t)=U^*(f_1 \cdot t,f_2 \cdot t)(x)$, $V^*(x,t)=V^*(f_1 \cdot 
	t,f_2 
	\cdot t)(x)$.    It is only necessary 
	to prove that $U^*(f_1 \cdot t,f_2 \cdot t)(x)$ and 
	$V^*(x,t)=V^*(f_1 \cdot t,f_2 
	\cdot t)(x)$ are uniformly almost periodic in $t\in\mathbb R$   with $x$ in 
	bounded 
	subsets 
	of 
	$\mathbb R$. Since $f_i(x,t,U,V)$ is uniformly  almost periodic  in $t$  
	with 
	$x\in \mathbb R$  and $(U,V)$ in bounded subsets of $\mathbb R^2$ for 
	$i=1,2$, 
	according to 
	Theorems 1.17 and 2.10 (\cite{fink1974}), for any 
	sequences $\{a_n\}
	\subset \mathbb R$ and $\{b_n\}\subset \mathbb R$, there exist 
	$\{t_n\}\subset 
	\{a_n\}$ and $\{s_n\}\subset \{b_n\}$  such that 
	\[\lim_{n\rightarrow \infty}f_i(x,t+t_n+s_n,U,V)=\lim_{n\rightarrow 
		\infty}\lim_{m\rightarrow \infty}f_i(x,t+t_n+s_m,U,V)\] for 
	$(x,t,U,V)\in 
	\mathbb 
	R^4, i=1,2$.  Assume that $$\lim_{n\rightarrow 
		\infty}f_i(x,t+t_n+s_n,U,V)=f^*_i(x,t,U,V),
	\lim_{m\rightarrow \infty}f_i(x,t+s_m,U,V)=f^{**}_i(x,t,U,V).$$ 
	Then we can get that 
	$$\lim\limits_{m\rightarrow\infty}U(x,t+s_m,U^*(f_1,f_2),V^*(f_1,f_2),f_1,f_2)=
	U^*(f^{**}_1\cdot t,f^{**}_2\cdot t)(x),$$
	$$\lim\limits_{m\rightarrow\infty}V(x,t+s_m,U^*(f_1,f_2),V^*(f_1,f_2),f_1,f_2)=
	V^*(f^{**}_1\cdot t,f^{**}_2\cdot t)(x)$$
	uniformly for $x$ in bounded sets of $\mathbb R$.
	Further, it follows that
	\begin{equation}
		\begin{aligned}
			&\lim_{n\rightarrow \infty}\lim_{m\rightarrow 
				\infty}U(x,t+t_n+s_m,U^*(f_1,f_2),V^*(f_1,f_2),f_1,f_2)
			\\&=\lim_{n\rightarrow\infty}U(x,t_n,U^*(f^{**}_1\cdot
			t,f^{**}_2\cdot t),V^*(f^{**}_1\cdot t,f^{**}_2\cdot 
			t),f^{**}_1\cdot 
			t,f^{**}_2\cdot t)
			\\&=\lim_{n\rightarrow\infty}U(x,t,U^*(f^{**}_1\cdot
			t_n,f^{**}_2\cdot 
			t_n),V^*(f^{**}_1\cdot
			t_n,f^{**}_2\cdot 
			t_n), f^{**}_1\cdot t_n,f^{**}_2\cdot t_n)
			\\&=U^*(f_1^{*}\cdot t,f_2^{*}\cdot 
			t)(x),
		\end{aligned}	
	\end{equation}
	
	\begin{equation}
		\begin{aligned}
			&\lim_{n\rightarrow \infty}\lim_{m\rightarrow 
				\infty}V(x,t+t_n+s_m,U^*(f_1,f_2),V^*(f_1,f_2),f_1,f_2)
			\\&=\lim_{n\rightarrow\infty}V(x,t_n,U^*(f^{**}_1\cdot
			t,f^{**}_2\cdot t),V^*(f^{**}_1\cdot t,f^{**}_2\cdot 
			t),f^{**}_1\cdot 
			t,f^{**}_2\cdot t)
			\\&=\lim_{n\rightarrow\infty}V(x,t,U^*(f^{**}_1\cdot
			t_n,f^{**}_2\cdot 
			t_n),V^*(f^{**}_1\cdot
			t_n,f^{**}_2\cdot 
			t_n),f^{**}_1\cdot t_n,f^{**}_2\cdot t_n)
			\\&=V^*(f_1^{*}\cdot t,f_2^{*}\cdot 
			t)(x)
		\end{aligned}	
	\end{equation}
	uniformly for $x$ in bounded sets of $\mathbb R$.
	Moreover,
	\[	\lim_{n\rightarrow 
		\infty}U(x,t+t_n+s_n,U^*(f_1,f_2),V^*(f_1,f_2),f_1,f_2)=U^*(f_1^{*}\cdot
	t,f_2^{*}\cdot 
	t)(x)\]
	\[	\lim_{n\rightarrow 
		\infty}V(x,t+t_n+s_n,U^*(f_1,f_2),V^*(f_1,f_2),f_1,f_2)=V^*(f_1^{*}\cdot
	t,f_2^{*}\cdot 
	t)(x)\]
	uniformly for $x$ in bounded sets of $\mathbb R$.
	Thus, \[\lim_{n\rightarrow \infty}\lim_{m\rightarrow 
		\infty}U(x,t+t_n+s_m,U^*({f_1,f_2}),V^*({f_1,f_2}),f_1,f_2)\]
	\[=\lim_{n\rightarrow
		\infty}U(x,t+t_n+s_n,U^*(f_1,f_2),V^*(f_1,f_2),f_1,f_2),\]
	\[\lim_{n\rightarrow \infty}\lim_{m\rightarrow 
		\infty}V(x,t+t_n+s_m,U^*({f_1,f_2}),V^*({f_1,f_2}),f_1,f_2)\]
	\[=\lim_{n\rightarrow
		\infty}V(x,t+t_n+s_n,U^*(f_1,f_2),V^*(f_1,f_2),f_1,f_2).\]
	According to the regularity and priori estimates for parabolic differential 
	equations, 
	$$U(x,t,U^*(f_1,f_2),V^*(f_1,f_2),f_1,f_2)\quad\text{and}\quad
	V(x,t,U^*(f_1,f_2),V^*(f_1,f_2),f_1,f_2)$$ are uniformly continuous for 
	$(x,t)\in\mathbb R^2$, applying Theorems 1.17 and 2.10 (\cite{fink1974}),  
	it follows that
	$U^*(f_1\cdot t,f_2\cdot t)(x)$ and $V^*(f_1\cdot 
	t,f_2\cdot t)(x)$ are almost periodic in $t\in \mathbb 
	R$ uniformly with  $x$ in bounded sets of $\mathbb R$.
	Therefore, our proof is completed.
\end{proof}

\begin{proof}[Proof of Theorem \ref{dichotomy}]
	Assume that (\textit{\textbf{H1}})-(\textit{\textbf{H5}}) hold. For any 
	given $g(0),h(0)$ and initial functions $(U_0,V_0)$ satisfying 
	(\ref{system-2}). Let 
	$$(U(x,t,U_0,V_0,g,h),V(x,t,U_0,V_0,g,h))$$ be the solution of system 
	(\ref{system-1}),
	It is easy to see that either $h_\infty-g_\infty<\infty$ or 
	$h_\infty-g_\infty=\infty$ holds. According to Theorem \ref{tm5}, if 
	$h_\infty-g_\infty<\infty$, then $h_\infty-g_\infty\leq 2L^*$. And 
	$$\lim\limits_{t\rightarrow\infty}
	(U(x,t;U_0,V_0,h_0),V(x,t;U_0,V_0,h_0))=0$$ uniformly for $x\in 
	[g_\infty,h_\infty]$. According to Theorem 
	\ref{tm6}, if $h_\infty-g_\infty=\infty,$ then 
	$$\lim\limits_{t\rightarrow 
		+\infty}U(x,t;U_0,V_0,h_0)-U^*(x,t)=0, \lim\limits_{t\rightarrow 
		+\infty}V(x,t;U_0,V_0,h_0)-V^*(x,t)=0$$ 
	locally 
	uniformly 
	for $x$ in 
	$\mathbb{R}$.
	Thus, the spreading-vanishing 
	dichotomy for system (\ref{system-1}) with (\ref{system-2}) holds.
\end{proof}
\begin{cor} According to 
	the above theorem, assume that 
	(\textit{\textbf{H1}})-(\textit{\textbf{H5}}) hold, it is natural to 
	obtain
	that  if $\lambda(t)<0$ for any $t>0$, then
	the trivial equilibrium
	$(0,0)$ 
	is globally asymptotically stable.
	If 
	$\lambda(T)>0$ for some $T>0$, then $h(T)-g(T)\geq2L^*$. Taking 
	$T$ as the 
	initial time, we can get that the disease will spread and the trivial 
	equilibrium $(0,0)$ is unstable.
\end{cor}

\begin{cor}
	According to the above arguments and the positivity of 
	$(U^*(x,t),V^*(x,t))$ in (\ref{ii}), by the persistence theory 
	in Section 3 of \cite{zhao2001} or upper and lower solution method, we can 
	get that when the spreading happens, there is a positive constant $c^*$ 
	such that 
	$$\lim\limits_{t\rightarrow \infty}U(x,t;U_0,V_0,h_0)\geq c^*, 
	\lim\limits_{t\rightarrow 
		\infty}V(x,t;U_0,V_0,h_0)\geq c^*.$$
	locally uniformly for $x\in \mathbb R$. 
\end{cor}

Finally, we turn to prove Theorem \ref{threshold}.
\begin{proof}[Proof of Theorem \ref{threshold}]
	\textbf{(1)} Assume that (\textit{\textbf{H5}}) holds, considering that 
	$h(t)$ is increasing and $g(t)$ is 
	decreasing, if 
	$\lambda(0)>0$,  then $h(0)-g(0)\geq2 L^{*}$, and
	\begin{equation}\label{equation-15}
		h_{\infty}-g_\infty>h(0)-g(0) \geq 2L^{*}.
	\end{equation}
	Further, we can get that
	$\lambda(A,\frac{h_\infty-g_\infty}{2})>0.$
	According to Theorem \ref{dichotomy}, we can obtain that
	$h_\infty-g_\infty=\infty$. Therefore, the disease is spreading.\\
	\textbf{(2)} Assume that $h(0)-g(0)<2L^{*}$. Denote 
	${h_\mu}(\infty):=\lim\limits_{t\rightarrow \infty}h_\mu(t), 
	{g_\mu}(\infty):=\lim\limits_{t\rightarrow \infty}g_\mu(t)$.
	
	Let
	\begin{equation}\label{nu}
		\Lambda:= \left\{\mu \mid {h_\mu}(\infty)-	
		{g_\mu}(\infty)<\infty\right\}, \nu:=\sup \Lambda.
	\end{equation}
	If $\Lambda$ is  an empty set, then 
	${h_{\mu^*}}(\infty)-{g_{\mu^*}}(\infty)=\infty$ for all $\mu>0$. In this 
	case,  $\mu^*=0$ satisfies the conditions.
	If $\Lambda$ is a
	nonempty set, we first prove that $\nu\in \Lambda$.
	On the contrary, assume that
	${h_{\nu}}(\infty)-{g_{\nu}}(\infty)=\infty.$ Then there exists a 
	$T>0$ 
	such that $h_{\nu}(T)-g_{\nu}(T)>2L^{*} .$ In view of the continuous 
	dependence of 
	$h_{\mu}$ and $g_\mu$ on $\mu,$ there is a $\varepsilon>0$ small enough 
	such 
	that 
	$h_{\mu}(T)-g_{\mu}(T)>2L^{*}$ for any $\mu \in\left[\nu-\varepsilon, 
	\nu+\varepsilon\right] .$ Therefore, we have
	$$
	{h_{\mu}}(\infty)-{g_{\mu}}(\infty)=\lim _{t \rightarrow \infty} 
	h_{\mu}(t)-g_\mu(t)>h_{\mu}(T)-g_{\mu}(T)>2L^{*}, \enspace \mu 
	\in\left[\nu-\varepsilon, 
	\nu+\varepsilon\right].
	$$
	According to (\ref{equation-15}), we obtain that
	${h_\mu}(\infty)-	
	{g_\mu}(\infty)=\infty,$ which implies that 
	$\Lambda\cap\left[\nu-\varepsilon, 
	\nu+\varepsilon\right]$ is an empty set.
	It is contradict to (\ref{nu}). Thus, we have 
	proved that $
	{h_{\nu}}(\infty)-{g_{\nu}}(\infty)<\infty$.
	
	When $\mu>\nu,$ we claim that ${h_\mu}(\infty)-	
	{g_\mu}(\infty)=\infty.$ 
	On the contrary, assume that ${h_\mu}(\infty)-	
	{g_\mu}(\infty)<\infty,$ then
	$\mu \leq \nu,$ which is a contradiction. Therefore, by Theorem 
	\ref{dichotomy}, the 
	spreading happens. 
	
	When 
	$\mu \leq \nu,$ by the Lemma \ref{hmu(t)}, we can obtain
	$$
	h_{\mu}(t)-g_{\mu}(t) \leq h_{\nu}(t)-g_{\nu}(t) 
	\text { for 
		all } t \in(0,+\infty).
	$$
	Moreover, ${h_\mu}(\infty)-	
	{g_\mu}(\infty)\leq 
	{h_{\nu}}(\infty)-	
	{g_{\nu}}(\infty)<\infty,$ thus, by Theorem \ref{dichotomy}, the vanishing 
	happens. In this case, we can take $\mu^*=\nu$.
	Therefore, our proof is completed.
\end{proof}

\begin{remark}
	{\rm When the initial infected domain is smaller than 2$L^*$, for any given 
		initial functions $(U_0,V_0),$ the spreading 
		or 
		vanishing of the epidemic disease mainly depend on the 
		front expanding rate $\mu$. }
\end{remark}

\noindent
\section{Simulations and Discussions}\label{discussion}	
\subsection{Simulations}
\noindent
\\In this subsection, we make some numerical simulations about our WNv model. 
Since the parameters $a_i(x,t)$ and $d_i(x,t)(i=1,2)$ are positive 
almost periodic functions, and the double boundaries are moving, the classical
numerical simulation methods is not proper. We use the implicit finite 
difference scheme developed 
in \cite{2009numerical} for numerical simulations about the free boundaries 
problems
to make some simulations to identify our results.

Fix the parameter values in system (\ref{system-1}) 
as follows, the explicit biological interpretations of which can be seen from 
Komar et al.\cite{komar2003}, Wonham et al.\cite{wonham2004epidemic} or Lewis 
et al.\cite{lewis2006traveling},
\[D_1=3,\enspace D_2=0.125,
\enspace N_1=1,
\enspace N_2=20, \enspace \beta=0.6,\]      
\[\alpha_1(x,t)=0.88\times(1+0.56\times \cos(\frac{t}{2}))
+0.088\times\dfrac{2+x}{1+x^2}\cos  x, x\in \mathbb R, t\geq0,\]
\[\alpha_2(x,t)=0.16\times(1+0.2\times \cos(\frac{\pi t}{3})
+0.024\times\dfrac{1+x}{1+x^2}\cos x, x\in \mathbb R, t\geq0,\]
\[d_1(x,t)=0.1\times(1+0.3\times \sin(\frac{t}{3}))
+0.02\times\dfrac{2+x}{1+x^2}\sin x, x\in \mathbb R, t\geq0,\]
\[d_2(x,t)=0.029\times(1+0.1\times \sin(\frac{\pi t}{2}))
+0.0016\times\dfrac{1+x}{1+x^2}\sin x, x\in \mathbb R, t\geq0.\]

Take the following initial functions:
\begin{equation}\label{61}
	U_{0}(x)=\left\{\begin{array}{ll}
		0.1\times \cos(\frac{\pi x}{2h_0}), & \text { for } -h_0 \leq x \leq 
		h_{0}, 
		\\
		0, & \text { for } |x|>h_{0}.
	\end{array}\right.
\end{equation}

\begin{equation}\label{62}
	V_{0}(x)=\left\{\begin{array}{ll}
		2\times \cos(\frac{\pi x}{2h_0}), & \text { for } -h_0 \leq x \leq 
		h_{0}, \\
		0, & \text { for } |x|>h_{0}.
	\end{array}\right.
\end{equation}	
Moreover, in order to simplify the simulations, we take $h(0)=h_0, g(0)=-h_0$. 
\subsubsection{The effect of the initial infected domain.}
\noindent
\\Fix $\mu=0.1$, take $h_0=2.0,1.0,0.6,0.5$, respectively. 
As is shown in Fig.1, the WNv is spreading  and the solution for system 
(\ref{system-1}) converges to a positive heterogeneous steady state 
in Fig.1(a) and Fig.1(b), which indicates that the recurrent apperance of the 
cases of infection;
while the  WNv is vanishing and the solution for system 
(\ref{system-1}) decays to 0 in Fig.1(c) and Fig.1(d). The Theorem 
\ref{threshold}~(1) implies that the disease will spread when $h(0)-g(0)\geq 
2L^*$.  For given initial functions $(U_0,V_0)$ in (\ref{61}) and 
(\ref{62}), according to  our 
simulations, we 
can deduce that the initial infected doamin threshold  $L^*\in (0.6,1.0)$.

Moreover, it indicates that the initial infected domain is larger, the 
eventual infected density of populations is bigger from Fig.1(a) and Fig.1(b), 
which conforms to the disease propagation mechanism. 
\begin{figure}[htbp]
	\centering
	\setlength{\abovecaptionskip}{0.cm}
	\subfigure[$h_0=2.0$.]{
		\begin{minipage}[t]{0.4\linewidth}
			\centering
			\includegraphics[width=2.2in]{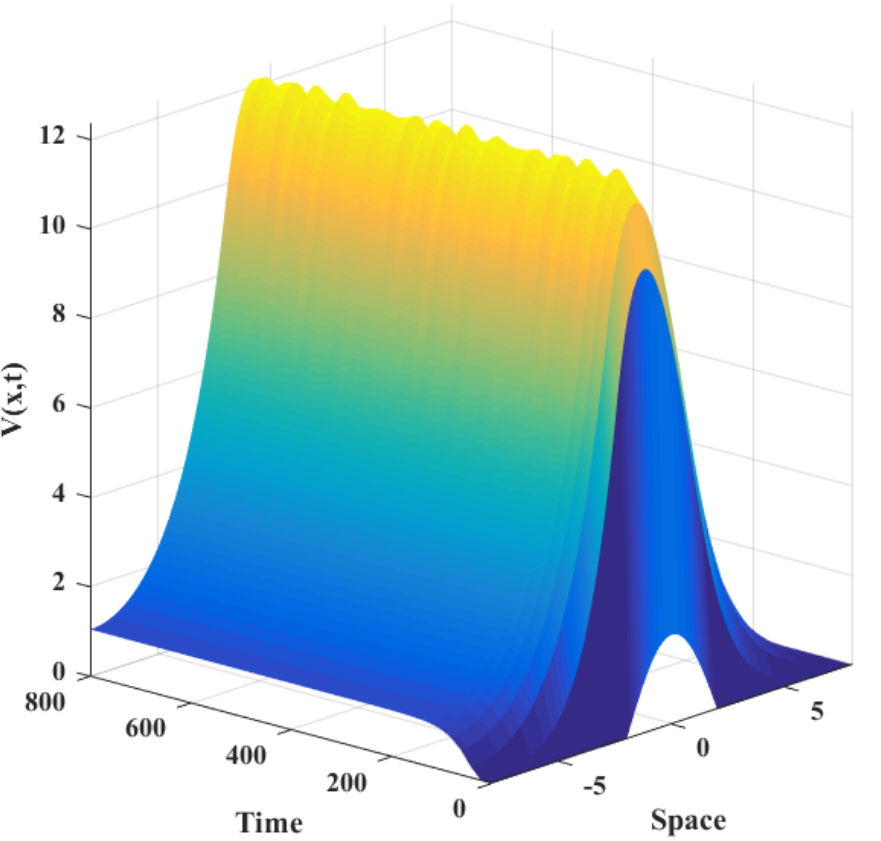}
			
		\end{minipage}%
	}\label{a}
	\subfigure[$h_0=1.0$.]{
		\begin{minipage}[t]{0.4\linewidth}
			\centering
			\includegraphics[width=2.2in]{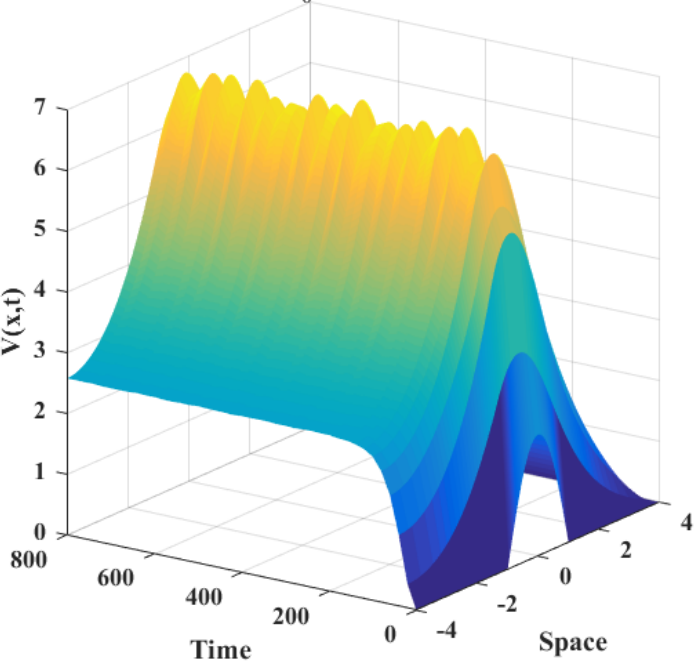}
		\end{minipage}%
	}\label{b}

	\quad
	\subfigure[$h_0=0.6$.]{
		\begin{minipage}[t]{0.4\linewidth}
			\centering
			\includegraphics[width=2.2in]{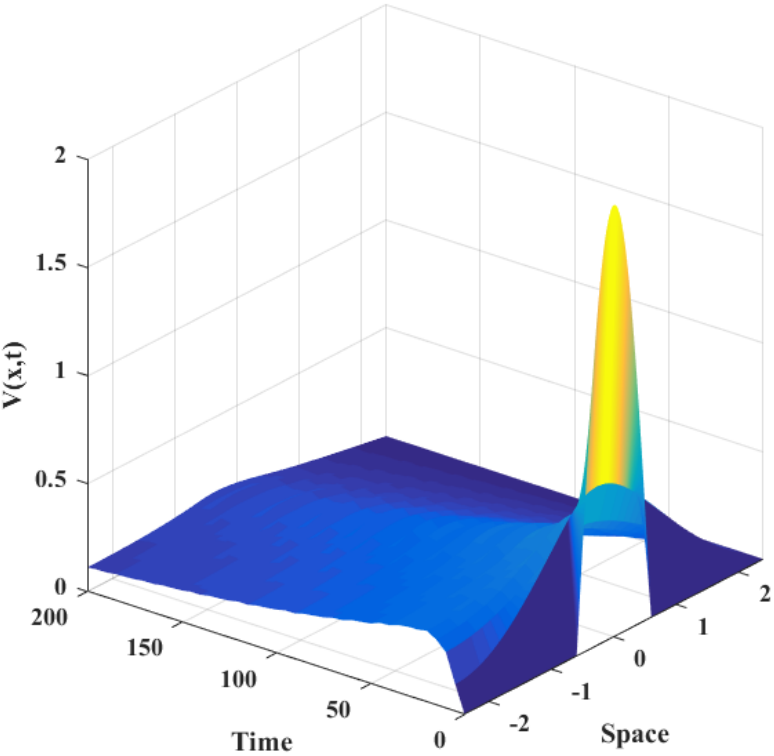}
		\end{minipage}%
	}%
	\subfigure[$h_0=0.5$.]{
		\begin{minipage}[t]{0.4\linewidth}
			\centering
			\includegraphics[width=2.2in]{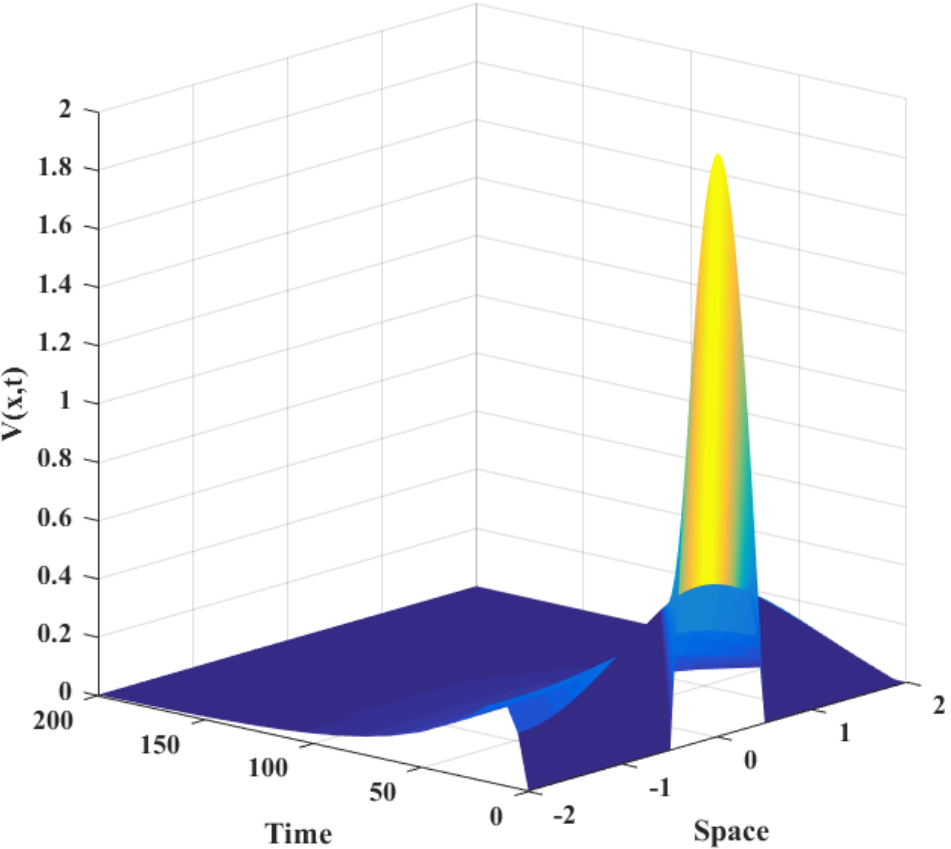}
		\end{minipage}%
	}%
	\centering
	\caption{Fix $\mu=0.1$, the Fig.1(a) and Fig.1(b) show that the 
		WNv is 
		spreading and the solution converges to a positive solution when 
		$h_0=2.0,1.0$, respectively, while the 
		Fig.1(c) and Fig.1(d) show that the WNv is 
		vanishing and the solution decays to 0 when $h_0=0.6,0.5$, 
		respectively. }
\end{figure}

\subsubsection{The effect of the infected domain boundary expanding rate.}
\noindent
\\Fix $h_0=0.6$, take $\mu=0.1,0.2$, respectively. As in shown in Fig.2, 
the WNv is spreading when $\mu=0.2$, while the WNv is vanishing when 
$\mu=0.1$. The Theorem \ref{threshold}(2) implies that the disease will spread 
when $\mu>\mu^*$ and the disease will vanish when $\mu\leq\mu^*$. For given  
initial functions 
$(U_0,V_0)$ and $h_0$, in view of 
our simulations,  we can deduce that the expanding capacity
rate threshold $\mu^*\in (0.1,0.2)$. According to this result, people can 
implement effective treatments to control the boundary expanding rate to 
restrain the propagation of the epidemic disease.

\begin{figure}[htbp]
	\centering
	\setlength{\abovecaptionskip}{0.cm}
	\subfigure[$\mu=0.2$.]{
		\begin{minipage}[t]{0.4\linewidth}
			\centering
			\includegraphics[width=2.2in]{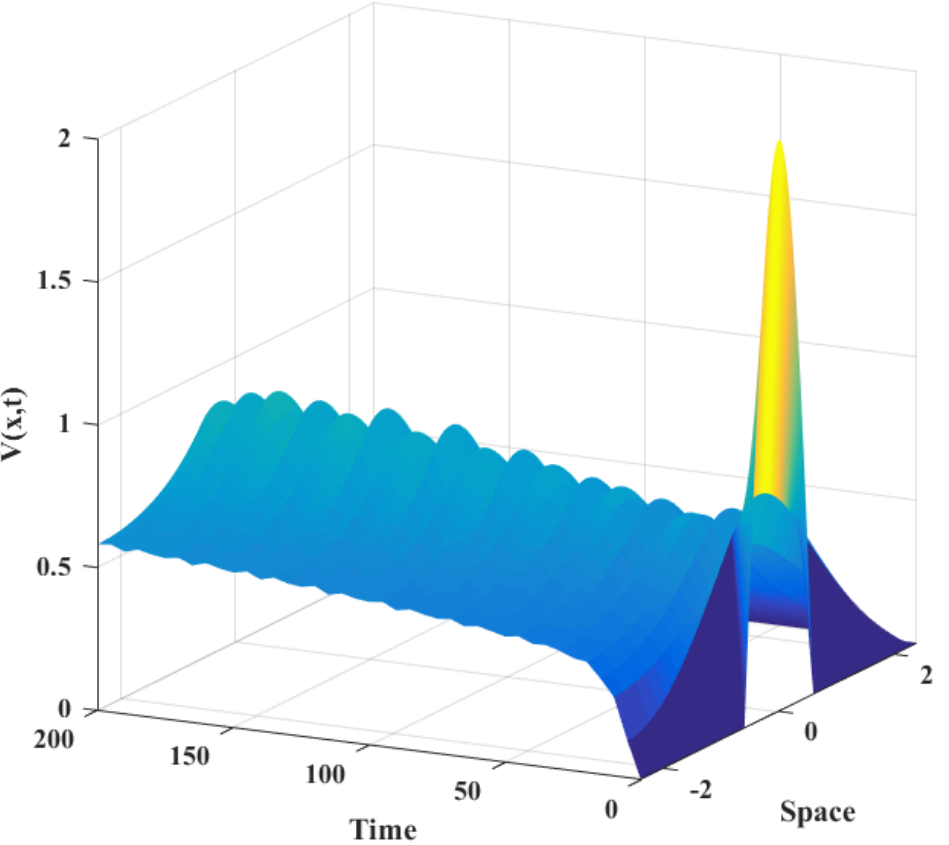}
		\end{minipage}%
	}%
	\subfigure[$\mu=0.1$.]{
		\begin{minipage}[t]{0.4\linewidth}
			\centering
			\includegraphics[width=2.2in]{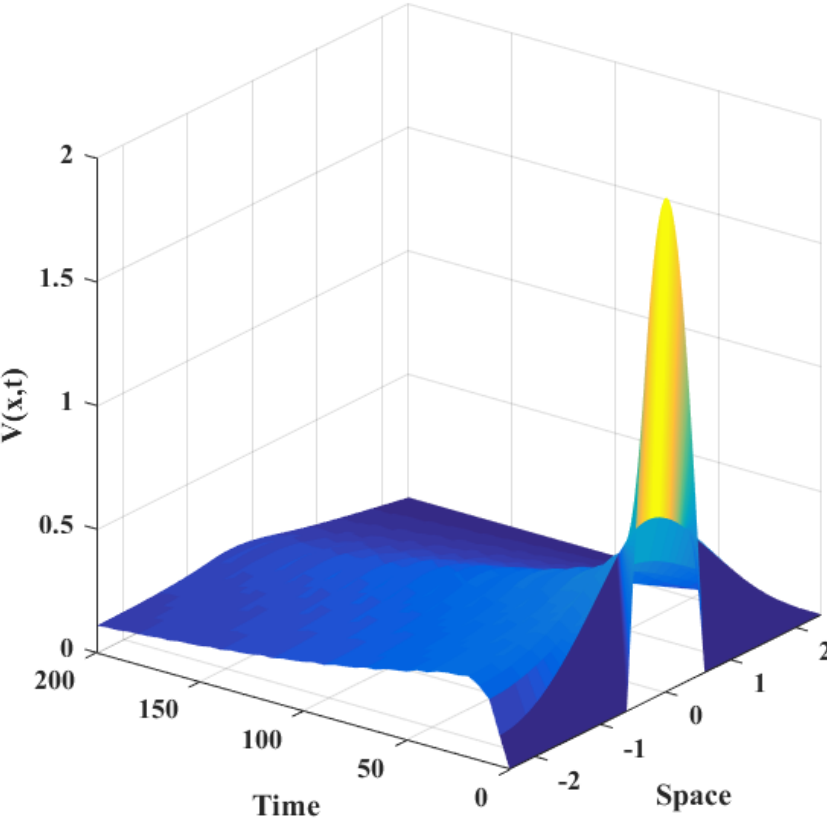}
		\end{minipage}%
	}%
	\centering
	\caption{Fix $h_0=0.6$,  the Fig.2(a) 
		shows that the WNv is spreading and the solution converges to a 
		positive 
		heterogeneous steady state when $\mu=0.2$, while 
		the Fig.2(b) shows that the WNv is vanishing and the solution decays to 
		0 when $\mu=0.1$. }
\end{figure}

\subsection{Discussions}	
\noindent
\\Neuroinvasive disease caused by West Nile virus is one of the most serious 
epidemic diseases and has brought considerable deaths since it occured. In 
order to 
supply feasible measures to predict and control the spreading of the epidemic 
disease, it is urgent to investigate the propagation mechanisms for WNv.
In order to describe
the transmission of WNv more reasonably, almost periodic mathematical 
biology models should be importantly considered. In this paper, we 
mainly propose a new reaction-diffusion WNv 
model (\ref{system-1}) with  moving infected domains $(g(t),h(t))$ in 
the spatial 
heterogeneous and time almost 
periodic environment and explore the long-time asymptotic dynamical behaviors  
of the solution for this model. 

Firstly, considering the spatial heterogeneity and time almost periodicity, we  
prove the global existence, uniqueness and get the regularity estimates of 
solution for system 
(\ref{system-1}), 
which is not trivial to obtain. 
Next, we define the principal Lyapunov exponent $\lambda(A,L)$ and $\lambda(t)$ 
with respect to 
time $t$ and 
get some analytic properties of it. Moreover,
we give the initial infected 
domain critical size $L^*$ using the 
principal Lyapunov exponent.  In this paper, under the assumption of
$\lambda(A,L)>0$ for 
$L\geq L^*$,  we obtain the following results: if $\lambda(t_0)>0$ for some 
$t_0\geq 0$, that 
is $h(t_0)-g(t_0)\geq2L^*$, 
then 
$h_\infty-g_\infty=\infty$ and the disease will spread 
no matter how big the diffusion 
rates and the initial data are; if $h(0)-g(0)<2L^*,$ 
there exists a threshold value $\mu^*\geq 0$ which represents the infected 
region expanding capacity. When $\mu>\mu^*,$ the disease will 
spread and the 
disease will vanish when $\mu\leq \mu^*$.
What is most important, assuming 
(\textit{\textbf{H1}})-(\textit{\textbf{H5}}), we obtain the 
long-time dynamical behaviors of WNv model by giving the 
spreading-vanishing dichotomy regimes of system (\ref{system-1}).  When the 
disease is vanishing, the 
densities $(U(x,t,g,h),V(x,t,g,h))$ of infected birds and mosquitoes will 
asymptotically converge to 0 uniformly for $x\in 
[g_\infty,h_\infty]$ and the eventually infected domain is no more than $2L^*$. 
When the 
disease is spreading, the 
densities $(U(x,t,g,h),V(x,t,g,h))$ of infected birds and mosquitoes will 
converge to a positive almost periodic solution 
$(U^*(x,t),V^*(x,t))$ of  system (\ref{system-10}) uniformly 
for $x$ in any compact 
subsets of $\mathbb R$. The asymptotic behavior of the solution when spreading 
occurs is largely 
different from the other homogeneous WNv models.
This result is advantageous to study the  cyclic outbreak laws of the West Nile 
virus 
caused by environmental differences and seasonal changes.

In view of the biological reality, our WNv model (\ref{system-1}) is first 
proposed  incorporate the spatial heterogeneity with time almost periodicity, 
which is more reasonable.
Meanwhile, we discuss the explicit dynamical behaviors by mathematical 
techniques,
which can be used to investigate other mosquito-borne epidemic models.
Moreover, our techniques in studying almost periodic 
systems different from other homogeneous and periodic systems 
can be applied in other almost periodic equations. Our methods using principal 
Lyapunov exponent can also applied to investigate other epidemic models.


\end{document}